\documentclass[a4paper,10.5pt]{article}
\usepackage{amsmath,amssymb,amsthm}
\usepackage{cite}
\usepackage[colorlinks, linkcolor=black,citecolor=black]{hyperref}
\usepackage{bm}
\usepackage[section]{algorithm}
\usepackage{cases}
\usepackage{graphicx,subfigure}
\usepackage{authblk}

\usepackage{enumitem} 
\newtheorem{mytheorem}{Theorem}[section]
\newtheorem{mylemma}{Lemma}[section]
\newtheorem{mycorollary}{Corollary}[section]

\theoremstyle{definition}

\newtheorem{myremark}{Remark}[section]
\newtheorem{mydefinition}{Definition}[section]

\numberwithin{equation}{section}
\numberwithin{figure}{section}
\numberwithin{table}{section}

\allowdisplaybreaks
\newcommand{\dif}{\mathrm{\,d}}

\newcommand{\DS}{\mathbb{D}}
\newcommand{\NS}{\mathbb{N}} 
\newcommand{\RS}{\mathbb{R}} 
\newcommand{\XS}{\mathbb{X}}
\newcommand{\YS}{\mathbb{Y}}

\newcommand{\tran}[1]{{#1}^\mathrm{T}}
\newcommand{\diag}{\mathrm{diag}}
\newcommand{\ball}{\bm{\mathrm{B}}}
\newcommand{\I}{\bm{\mathrm{I}}}

\renewcommand{\vec}[1]{\bm{\mathrm{#1}}}
\newcommand{\vf}{\vec{f}}

\begin{document}

\title{Semilocal Convergence Analysis for Two-Step Newton Method
     under Generalized Lipschitz Conditions in Banach Spaces\thanks{
     This work was supported in part by the Fujian Province National Science Foundation of China (Grant 2016J05015).}}
\author{Yonghui Ling\footnote{Corresponding author.\\ \indent E-mail address: yhling@mnnu.edu.cn (Y. Ling).} }

\author{Juan Liang}

\affil{Department of  Mathematics, Minnan Normal University, Zhangzhou 363000, China}

\maketitle

\begin{abstract}
In the present paper, we consider the semilocal convergence
problems of the two-step Newton method for solving nonlinear operator
equation in Banach spaces. Under the assumption that the first derivative
of the operator satisfies a generalized Lipschitz condition, a new
semilocal convergence analysis for the two-step Newton method is presented.
The Q-cubic convergence is obtained by an additional condition.
This analysis also allows us to obtain three important spacial cases about
the convergence results based on the premises of Kantorovich, Smale and
Nesterov-Nemirovskii types. An application of our convergence results is to the approximation
of minimal positive solution for a nonsymmetric algebraic Riccati equation
arising from transport theory.
\end{abstract}

\section{Introduction}

In this paper, we aim to study the convergence of iterative methods for approximating
the solution of the nonlinear operator equation
\begin{equation}
\label{eq:F(x)=0}
F(x) = 0,
\end{equation}
where $F$ is a given Fr\'{e}chet differentiable nonlinear operator which maps from some open
convex subset $\DS$ in a Banach space $\XS$ to another Banach space $\YS$.
Newton's method is probably the most important and efficient iterative method known for solving such equation,
which proceeds as follows: for a given initial point $x_0 \in \DS$, construct iteratively a sequence such that
\begin{equation}
\label{it:NM}
x_{k+1} = x_k - F'(x_k)^{-1} F(x_k), \quad k = 0, 1, 2, \ldots.
\end{equation}

An important convergence result for Newton's method \eqref{it:NM} is the well-known Kantorovich theorem \cite{KantorvichA1982},
which gives a simple and clear criterion guaranteeing the existence of a solution of equation \eqref{eq:F(x)=0},
uniqueness of this solution in a prescribed ball and the quadratic convergence of Newton's method \eqref{it:NM}.
The assumptions used essentially focus on the nonsingularity of the first Fr\'{e}chet derivative of $F$ at initial point $x_0$
and the behavior of the Fr\'{e}chet derivative of $F$ on an appropriate metric ball of
the initial point $x_0$ contained in $\DS$ (such as that the second Fr\'{e}chet derivative of $F$ is bounded or
the first Fr\'{e}chet derivative of $F$ is Lipschitz continuous). There are a lot of works on the weakness
and/or extension of the assumptions made on the operator $F$ and its derivative,  see for example
\cite{DeuflhardH1979,GraggT1974,GutierrezH2000,EzquerroH2002,EzquerroHM2018,Huang1993,Argyros2001,ArgyrosK2015}.
For more studies and applications of the Kantorovich like theorem, one can see the recent survey paper
by Kelley in \cite{Kelley2018} for more details.

Another important convergence result concerning Newton's method \eqref{it:NM} is the famous Smale $\alpha$-theory (with analytic $F$),
presented in the report \cite{Smale1987} (see also \cite{Smale1986,Smale1997}),
where the notion of approximate zeros was proposed and the rules were established to determine if an initial point $x_0$ is an approximate zero.
Since then, this line of research has been extensively studied resulting many significant improvements
and extensions in several directions; see for example
\cite{Chen1994,DedieuK2002,DedieuPM2003,DedieuS2000a,DedieuS2000b,ShubS1996,WangH1990,WangL2001,LiN2016}
and references therein. In particular, to drop the analytic assumption, Wang and Han
in \cite{WangH1997a} (see also \cite{Wang1999})
introduced the weak $\gamma$-condition for nonlinear twice continuously differentiable operator $F$
between Banach spaces. The notion of the weak $\gamma$-condition was also used in
\cite{LiW2006} and \cite{LiN2013} to extend and improve the corresponding results in
\cite{DedieuPM2003} and \cite{Dedieu2000}, respectively. Recently, the main results
obtained in \cite{LiN2013} were improved further by the corresponding ones in \cite{LiN2018}.

Wang in \cite{Wang1999} introduced some generalized Lipschitz condition called Lipschitz condition with $L$-average,
under which Kantorovich like convergence criteria and Smale's $\alpha$-theory can be investigated together.
The generalized Lipschitz condition was also used to study the convergence for various iterative methods.
For example, Xu and Li in \cite{XuL2008} investigated the semilocal convergence of Gauss-Newton's method
for singular systems with constant rank derivatives under the hypothesis that the derivatives satisfy the generalized Lipschitz condition.
This notion was extended to study the convergence of Gauss-Newton's method for one kind of special
singular systems of equations in \cite{LiHW2010} and improved further in \cite{ArgyrosH2011b}.
Besides, Alvarez et al. in \cite{AlvarezBM2008} extended this notion to a Riemannian contest and
established a unified convergence result for Newton's method. 
Recently, Ferreira and Svaiter in \cite{FerreiraS2009} presented a new
convergence analysis for Kantorovich's theorem which makes clear, with respect to Newton's method, the
relationship of the majorizing function and the nonlinear operator under consideration.
More extensions of this idea are referred to \cite{FerreiraS2012,ferreira2015,BittencourtF2017}.

As is well known, there are several kinds of cubic generalization for Newton's method.
The most important two are the Euler method and the Halley method; see for example
\cite{CandelaM1990a,CandelaM1990b,HanW1997,EzquerroH2005,EzquerroH2009a,EzquerroH2009b,LingX2014} and references therein.
For the applications in the field of matrix functions,
the efficiency (when properly implemented) of these two methods have been shown
in \cite{Iannazzo2008,LingH2017} for computing matrix $p$th root and
in \cite{NakatsukasaBG2010,NakatsukasaF2016} for computing the
polar decomposition of a matrix.
Another more general family of the cubic extensions is the family of Euler-Halley methods
type methods in Banach spaces, which includes the Euler and the Halley method as its
special cases and has been studied extensively in \cite{GutierrezH1997,Han2001}.

For general nonlinear equation \eqref{eq:F(x)=0}, however, the preceding classical third-order
methods need to evaluate the second Fr\'{e}chet derivative which is very time consuming.
The order of convergence of the classical two-step Newton method has also three,
but without evaluating any second Fr\'{e}chet derivative. The two-step Newton method
with initial point $x_0$ is defined by
\begin{equation}
\label{it:TwoStepNM}
\left\{
\begin{aligned}
& y_k = x_k - F'(x_k)^{-1}F(x_k), \\
& x_{k+1} = y_k - F'(x_k)^{-1}F(y_k),
\end{aligned}\right.
\quad k = 0,1,2,\ldots.
\end{equation}
The results concerning semilocal convergence (including existence, uniqueness and convergence)
of this iterative method have been studied under the
assumptions of Newton-Kantorovich type. By applying the majorizing function technique used
in the previous work of Zabrejko and Nguen \cite{ZabrejkoN1987} on Newton’s method,
Appell et al. \cite{AppellDEZ1995} established the
semilocal convergence and error estimate under the assumption that the first Fr\'{e}chet derivative
satisfies the Lipschitz condition.
Amat et al. \cite{AmatBG2011} investigated the convergence behavior based fundamentally on a
generalization required to the second Fr\'{e}chet derivative of $F$.
Recently, to weak the conditions used in \cite{AmatBG2011},
Magre\~{n}\'{a}n Ruiz and Argyros in \cite{ArgyrosMR2014} presented new convergence
analysis (including semilocal and local convergence) under the hypotheses that the first Fr\'{e}chet
derivative of $F$ satisfies Lipschitz and Lipschitz-like conditions.

The motivation of this paper is based on the following  two aspects of applications for
the two-step Newton method \eqref{it:TwoStepNM}. The first one stems from \cite{LinB2008, LingX2017}
for solving a nonsymmetric algebraic Riccati equation arising in transport theory,
where the monotone convergence guaranteeing the implementation of the two-step Newton based algorithm was showed.
The second one stems from \cite{ChenWS2018} for the inverse eigenvalue
problems (IEP), where the two-step Newton method \eqref{it:TwoStepNM} was used to present effective
algorithms for solving the solution of the IEP.

The goal of this paper is to establish a general semilocal convergence result for the two-step
Newton method \eqref{it:TwoStepNM} under the assumption that the first derivative of $F$ satisfies
some generalized Lipschitz condition,
which was introduced by Wang in \cite{Wang1999} for Newton's method \eqref{it:NM}.
When the unified convergence criterion given in \cite{Wang1999} is satisfied, we show that the
existence and uniqueness of a solution, together with the Q-superquadratic convergence of the two-step
Newton method \eqref{it:TwoStepNM}.
In our convergence analysis, the relationships between the majorizing function and the nonlinear operator
are made clear. Moreover, we show also that the two-step Newton method \eqref{it:TwoStepNM}
is Q-cubically convergent under a slightly stronger condition.
In particular, this convergence analysis allows us to obtain some important special cases,
which include Kantorovich-type convergence result under the Lipschitz condition, Smale-type convergence
results under the $\gamma$-condition and the convergence result for self-concordant functions
under the Nesterov-Nemirovskii condition. We also adapt our convergence result to
compute the approximation of minimal positive solution
of a nonsymmetric algebraic Riccati equation arising from transport theory. Numerical experiments
confirm our convergence result.

The rest of this paper is organized as follows.
In Section 2, we introduce some preliminary notions and properties of the majorizing function and majorizing sequences.
The main result about the semilocal convergence and error estimate are stated in Section 3.
In Section 4 we provide the convergence analysis for the main result.
We extend our convergence results in Section 5 to compute the minimal positive solution
of a nonsymmetric algebraic Riccati equation arising from transport theory.
We conclude with some final remarks in Section 6.

\section{Preliminaries}

Let $\XS$ and $\YS$ be Banach spaces. For $x \in \XS$ and a positive number $r$,
throughout the whole paper, we use $\ball(x,r)$ to stand for the open ball with radius $r$
and center $x$, and $\overline{\ball(x,r)}$ denote its closure. Moreover, $\I$ denotes the identity operator.

We assume that
$L(\cdot)$ is a positive nondecreasing function on $[0,R)$, where $R > 0$ satisfies
\begin{equation}
\label{cons:R}
\frac{1}{R} \int_0^R L(u)(R - u) \dif{u} = 1.
\end{equation}
Let $\beta > 0$.
The majorizing function $h: [0, R] \to \RS$ is defined by
\begin{equation}
\label{fun:h(t)}
h(t) = \beta - t + \int_0^t L(u)(t - u) \dif{u}, \quad t \in [0,R].
\end{equation}
This majorizing function was introduced by Wang in \cite{Wang1999} to study
the semilocal convergence of Newton's method \eqref{it:NM}. Clearly, we have
$$
h'(t) = - 1 + \int_0^t L(u) \dif{u}, \quad t \in [0,R)
$$
and
$$
h''(t) = L(t) > 0 \quad \text{for a.e. } t \in [0,R).
$$
Then, we obtain that
$$
\int_s^t L(u) \dif{u} = h'(t) - h'(s) \quad \text{for any } s, t \in [0, R) \text{ with } s < t.
$$
This simple equality will be frequently applied to our convergence analysis for the two-step Newton method \eqref{it:TwoStepNM}.
Assume that $r_0$ satisfies
\begin{equation}
\label{cons:r0}
\int_0^{r_0} L(u) \dif{u} = 1.
\end{equation}
It follows that $h(t)$ is strictly convex, $h'(t)$ increasing, convex and $-1 \leq h'(t) < 0$ for any $t \in [0,r_0)$.

The following lemma gives some properties about elementary convex analysis and will also be frequently used
in our convergence analysis for the two-step Newton method \eqref{it:TwoStepNM}.

\begin{mylemma}
\label{lem:ConvexProperty}
Let $R > 0$. If $f: (0,R) \to \RS$ is continuously differentiable and convex, then
\begin{itemize}[font=\upshape]
\item[(i)]
$(1 - \theta) f'(\theta t) \leq \cfrac{f(t) - f(\theta t)}{t} \leq (1 - \theta) f'(t)$ for all $t \in (0, R)$ and $0 \leq \theta \leq 1$.
\item[(ii)]
$\cfrac{f(u) - f(\theta u)}{u} \leq \cfrac{f(v) - f(\theta v)}{v}$ for all $u, v \in (0,R), u < v$ and $0 \leq \theta \leq 1$.
\end{itemize}
In particular, if $f$ is strictly convex, then the above inequalities are strict.
\end{mylemma}

\begin{proof}
See Theorem 4.1.1 and Remark 4.1.2 in the book of Hiriart-Hrruty and Lemar\'{e}chal
\cite[p.21]{HiriartL1993}.
\end{proof}

Define
\begin{equation}
\label{cons:b}
b := \int_0^{r_0} L(u) u \dif{u}.
\end{equation}
The following lemma is taken from \cite[Lemma 1.2]{Wang1999} which gives some basic properties
for the majorizing function $h$.

\begin{mylemma}[\cite{Wang1999}]
\label{lem:h(t)_property}
If $0 < \beta < b$, then $h$ is decreasing monotonically in $[0,r_0]$, while it is increasing
monotonically in $[r_0, R]$ and
\begin{equation}
\label{ineq:h_beta_property1}
h(\beta) > 0,  \quad h(r_0) = \beta - b < 0, \quad h(R) = \beta > 0.
\end{equation}
Moreover, $h$ has a unique zero in each interval, denoted by $t^*$ and $t^{**}$. They satisfy
\begin{equation}
\label{ineq:h_beta_property2}
\beta < t^* < \frac{r_0}{b} \cdot \beta < r_0 < t^{**} < R.
\end{equation}
\end{mylemma}

Let $\{s_k\}$ and $\{t_k\}$ denote the corresponding sequences generated by the two-step
Newton method for the majorizing function $h$ with the initial point $t_0 = 0$, that is,
\begin{equation}
\label{it:TwoStepNM_Scalar}
\left\{
\begin{aligned}
& s_k = t_k - \frac{h(t_k)}{h'(t_k)}, \\
& t_{k+1} = s_k - \frac{h(s_k)}{h'(t_k)},
\end{aligned}
\quad k = 0,1,2,\ldots.
\right.
\end{equation}

The lemma below describes the convergence property of the sequences
$\{s_k\}$ and $\{t_k\}$, which is crucial for the semilocal convergence
analysis of the two-step Newton method \eqref{it:TwoStepNM}.

\begin{mylemma}
\label{lem:TwoStepNM_Scalar_Conv}
Let $\{s_k\}$ and $\{t_k\}$ be the sequences generated by \eqref{it:TwoStepNM_Scalar}.
Suppose that $0 < \beta \leq b$. Then we have
\begin{equation}
\label{ineq:tk<sk<tk+1}
0 \leq t_k < s_k < t_{k+1} < t^* \quad \text{for all } k \geq 0.
\end{equation}
Moreover, $\{s_k\}$ and $\{t_k\}$ converge increasingly to the same point $t^*$.
\end{mylemma}

\begin{proof}
To show that \eqref{ineq:tk<sk<tk+1} holds for the case $k=0$, we note that $0 = t_0 < s_0 = \beta$
and that
$$
t_1 = s_0 - \frac{h(s_0)}{h'(t_0)} = \beta+ h(\beta).
$$
By \eqref{ineq:h_beta_property1}, we have $t_1 > \beta = s_0$. It remains to show that
$t_1 < t^*$ for the case $k=0$. To this end, we define a real function $\Phi(t)$ in $(0, r_0)$ by
$$
\Phi(t) := t + h(t), \quad t \in (0, r_0).
$$
Then
$$
\Phi'(t) = 1 + h'(t) = \int_0^t L(u)\dif{u} > 0, \quad t \in (0, r_0).
$$
That is, $\Phi(t)$ is increasing monotonically in $(0, r_0)$. It follows from \eqref{ineq:h_beta_property2} that
$\Phi(\beta) < \Phi(t^*)$, which implies that $t_1 < t^*$. Hence \eqref{ineq:tk<sk<tk+1} holds for the case $k=0$.

Now, we assume that
$$
0 \leq t_{k-1} < s_{k-1} < t_k < t^* \quad \text{for some } k \geq 1.
$$
By Lemma \ref{lem:h(t)_property}, we have $h(t) \geq 0$ for each $t \in [0,t^*]$
and $h(t_k)/h'(t_k) < 0$. The later one means that $s_k > t_k$. Define
$$
N(t) := t - \frac{h(t)}{h'(t)}, \quad t \in [0,t^*].
$$
Then
$$
N'(t) = \frac{h(t) h''(t)}{h'(t)^2} > 0 \quad \text{for a.e. } t \in [0,t^*],
$$
which implies that $N(t)$ is increasing monotonically on $[0,t^*]$. Hence, we have
$$
s_k = t_k - \frac{h(t_k)}{h'(t_k)} < t^* - \frac{h(t^*)}{h'(t^*)} = t^*.
$$
It follows that $h(s_k)/h'(t_k) < 0$ and so $t_{k+1} > s_k$.

Moreover, since $h$ is strictly convex in $[0,r_0)$ and $h'$ increasing monotonically on $[0,t^*]$,
one has from Lemma \ref{lem:ConvexProperty} that
$$
h'(t_k) < h'(s_k) < \frac{h(t^*) - h(s_k)}{t^* - s_k} < 0,
$$
which gives
$$
\frac{h(t^*) - h(s_k)}{h'(t_k)} < t^* - s_k.
$$
Note that $h'(t_k) < h'(t^*)$, we have
$$
\frac{h(t^*)}{h'(t^*)} - \frac{h(s_k)}{h'(t_k)}
    < \frac{h(t^*)}{h'(t_k)} - \frac{h(s_k)}{h'(t_k)}
    < t^* - s_k.
$$
This implies that
$$
t_{k+1} = s_k - \frac{h(s_k)}{h'(t_k)} < t^* - \frac{h(t^*)}{h'(t^*)} = t^*.
$$
Therefore, \eqref{ineq:tk<sk<tk+1} holds for all $k \geq 0$ by mathematical induction. The inequalities in
\eqref{ineq:tk<sk<tk+1} imply that $\{s_k\}$ and $\{t_k\}$ converge increasingly to some same point,
say $\zeta$. Clearly, $\zeta \in [0,t^*]$ and $\zeta$ is a zero of $h$ on $[0,t^*]$. Noting that
$t^*$ is the unique zero of $h$ in $[0,r_0)$ by Lemma \ref{lem:h(t)_property},
one has that $\zeta = t^*$. The proof is complete.
\end{proof}

we conclude this section with the notions of generalized Lipschitz condition and Q-order of convergence.

\begin{mydefinition}
\label{def:LAverageLipschitzCond}
Let $x_0 \in \DS$ be such that $F'(x_0)^{-1}$ is nonsingular and $r > 0$ such that $\ball(x_0, r) \subseteq \DS$.
Then, $F'$ is said to satisfy the $L$-average Lipschitz condition on $\ball(x_0, r)$ if, for any
$x, y \in \ball(x_0, r)$ with $\|x - x_0\| + \|y - x\| < r$,
\begin{equation}
\label{cond:LAverageLipCond}
\|F'(x_0)^{-1}[F'(y) - F'(x)]\| \leq \int_{\|x - x_0\|}^{\|x - x_0\| + \|y - x\|} L(u) \dif{u}.
\end{equation}
\end{mydefinition}

The preceding generalized Lipschitz condition was first introduced by
Wang in \cite{Wang1999} where the terminology
of ``the center Lipschitz condition in the inscribed sphere with $L$-average'' was used.
Subsequently, to study the convergence behavior of Gauss-Newton, some modified versions
were introduced by Li and Ng in \cite{LiN2007} for convex composite optimization
and Li et al. in \cite{LiHW2010} for singular systems of equations.

\begin{mydefinition}
\label{def:Q_Order_Conv}
Let sequence $\{x_k\} \subset \XS$. We say that
$\{x_k\}$ converges to $x^*$ with Q-superquadratic if,
for any $c > 0$, there exists a constant $N_c \geq 0$ such that
$\|x_{k+1} - x^*\| \leq c \|x_k - x^*\|^2$ holds for all $k \geq N_c$, or equivalently
$$
\lim_{k \to \infty} \frac{\|x_{k+1} - x^*\|}{\|x_k - x^*\|^2} = 0.
$$
In addition, we say that $\{x_k\}$ converges to $x^*$ with Q-cubic if there exist
two constants $c \geq 0$ and $N_c \geq 0$ such that
$\|x_{k+1} - x^*\| \leq c \|x_k - x^*\|^3$ holds for all $k \geq N_c$, or equivalently
$$
\limsup_{k \to \infty} \frac{\|x_{k+1} - x^*\|}{\|x_k - x^*\|^3} < \infty.
$$
\end{mydefinition}

Q-order of convergence is well-known concept that measure the speed of convergence of sequences.
One can see \cite{Potra1989,Jay2001} for more properties on this notion.

\section{The main theorem and corollaries}

In this section, we present the main semilocal convergence result of this paper
for the two-step Newton method \eqref{it:TwoStepNM} under $L$-average Lipschitz
condition \eqref{cond:LAverageLipCond} in Banach spaces. Then, we obtain three important
special cases from this main result. They include the Kantorovich-type convergence result
under the Lipschitz condition, Smale-type convergence
result for analytical operators and the convergence result
for self-concordant functions under the Nesterov-Nemirovskii condition.

Let $x_0 \in \DS$ be the initial point such that the inverse $F'(x_0)^{-1}$ exists and let
$\ball(x_0, r_0) \subset \DS$, where $r_0$ satisfies \eqref{cons:r0}. Set
\begin{equation}
\label{cons:beta}
\beta := \|F'(x_0)^{-1}F(x_0)\|.
\end{equation}
Recall that $b$ is given by \eqref{cons:b}, $t^*$ and $t^{**}$ are the unique zeros of
the majorizing function $h$ (see \eqref{fun:h(t)}) in $[0, r_0]$ and $[r_0, R]$, respectively,
where $R$ satisfies \eqref{cons:R}. Recall also that $\{t_k\}$ is the sequence
generated by \eqref{it:TwoStepNM_Scalar}.

\begin{mytheorem}
\label{th:SemilocalConv}
Let $F: \DS \subset \XS \to \YS$ be a continuously Fr\'{e}chet differentiable nonlinear operator,
$\DS$ open and convex. Assume that there exists an initial point $x_0 \in \DS$ such that
$F'(x_0)^{-1}$ exists and that $F'$ satisfies the $L$-average Lipschitz condition \eqref{cond:LAverageLipCond}
on $\ball(x_0, t^*)$.
Let $\{x_k\}$ be the sequence generated by the two-step Newton method \eqref{it:TwoStepNM}
with initial point $x_0$. If $0 < \beta \leq b$, then $\{x_k\}$ is well-defined and converges
Q-superquadratically to a solution
$x^* \in \overline{\ball(x_0, t^*)}$ of \eqref{eq:F(x)=0}, and this solution $x^*$
is unique in $\overline{\ball(x_0, r)}$, where $t^* \leq r < t^{**}$. Moreover, if
\begin{equation}
\label{cond:CubicConv_Cond}
2 + \frac{t^*h''(t^*)}{h'(t^*)} > 0 \iff 2 - \frac{t^*L(t^*)}{1 - \int_0^{t^*} L(u) \dif{u}} > 0,
\end{equation}
then the order of convergence is cubic at least and we have the following error bounds
\begin{equation}
\label{estimate:norm_x*-xk+1_LAverageLipCond}
\|x^* - x_{k+1}\|
    \leq \frac{1}{2}H_*^2
        \cdot \frac{2 - t^*H_*}{2 + t^*H_*}
        \cdot \|x^* - x_k\|^3, \quad k \geq 0,
\end{equation}
where $H_* \triangleq h''(t^*)/h'(t^*)$.
\end{mytheorem}

\begin{myremark}
The convergence criterion $0 < \beta \leq b$ given in Theorem \ref{th:SemilocalConv} 
was obtained by Wang in \cite{Wang1999} for studying the quadratic convergence 
of Newton's method \eqref{it:NM} under a unified framework. As is stated in Theorem \ref{th:SemilocalConv},
this criterion guarantees only superquadratic convergence for the two-step
Newton method \eqref{it:TwoStepNM}. To obtain cubic convergence, we also need
the condition \eqref{cond:CubicConv_Cond}.
\end{myremark}

The proof of Theorem \ref{th:SemilocalConv} will be presented in Section 4.
In what follows, based on Theorem \ref{th:SemilocalConv}, we will obtain some corollaries by taking various forms
of the positive function $L$.

Firstly, for the case when $L$ is a positive constant function, then the $L$-average Lipschitz condition
\eqref{cond:LAverageLipCond}
reduces to the following affine-invariant Lipschitz condition:
\begin{equation}
\label{cond:LipCond}
\|F'(x_0)^{-1}[F'(y) - F'(x)]\| \leq L\|y - x\|, \quad x, y \in \ball(x_0, r_0),
\end{equation}
where $r_0 = 1/L$ due to \eqref{cons:r0}. The majorizing function $h$ defined by \eqref{fun:h(t)}
now has the form below:
$$
h(t) = \beta - t + \frac{L}{2} t^2, \quad t \in [0, R],
$$
where $R = 2/L$ by \eqref{cons:R}. The constant $b$ defined in \eqref{cons:b} reduces to
$b = 1/(2L)$. Moreover, thanks to Lemma \ref{lem:h(t)_property}, if $L\beta < 1/2$,
then the zeros of $h$ in $(0, 1/L)$ and $(1/L, 2/L)$ are
\begin{equation}
\label{cons:t*t**_LipCond}
t^* = \frac{1 - \sqrt{1-2L\beta}}{L} \quad \text{and} \quad t^{**} = \frac{1 + \sqrt{1-2L\beta}}{L},
\end{equation}
respectively. Therefore, we have the following Kantorovich-type convergence result from Theorem \ref{th:SemilocalConv} for
two-step Newton method \eqref{it:TwoStepNM} under Lipschitz condition \eqref{cond:LipCond}.

\begin{mycorollary}
\label{cor:Conv_LipCond}
Let $F: \DS \subset \XS \to \YS$ be a continuously Fr\'{e}chet differentiable nonlinear operator,
$\DS$ open and convex. Assume that there exists an initial point $x_0 \in \DS$ such that
$F'(x_0)^{-1}$ exists and that $F'$ satisfies the Lipschitz condition \eqref{cond:LipCond}.
Let $\{x_k\}$ be the sequence generated by the two-step Newton method \eqref{it:TwoStepNM}
with initial point $x_0$. If $0 < L\beta \leq 1/2$, then $\{x_k\}$ is well-defined and converges
Q-superquadratically to a solution
$x^* \in \overline{\ball(x_0, t^*)}$ of \eqref{eq:F(x)=0}, and this solution $x^*$
is unique in $\overline{\ball(x_0, r)}$, where $t^* \leq r < t^{**}$,
$t^*$ and $t^{**}$ are given in \eqref{cons:t*t**_LipCond}.
Moreover, if $0 < L\beta < 4/9$, then the order of convergence is cubic at least and
we have the following error bounds
$$
\|x^* - x_{k+1}\|
    \leq \frac{L^2}{2(1-2L\beta)}
        \cdot \frac{\sqrt{1-2L\beta} + 1}{3\sqrt{1-2L\beta} - 1}
        \cdot \|x^* - x_k\|^3, \quad k \geq 0.
$$
\end{mycorollary}

\begin{myremark}
The convergence result in Corollary \ref{cor:Conv_LipCond} is obtained under the weaker assumption on
the Lipschitz condition in comparison with the one presented in \cite{AmatBG2011}, where the assumption
that the second derivative satisfies the Lipschitz condition is needed.
To obtain the Q-cubic convergence, we need the convergence criterion $0 < L\beta < 4/9$,
which has slightly stronger than the usual one $0 < L\beta < 1/2$ for ensuring the quadratic 
convergence of  Newton's method \eqref{it:NM}.
\end{myremark}

Secondly, we suppose that $\gamma > 0$. Let $L$ be the positive function defined by
\begin{equation*}
L(u) := \frac{2\gamma}{(1 - \gamma u)^3}, \quad u \in [0, \frac{1}{\gamma}).
\end{equation*}
Then, the $L$-average Lipschitz condition \eqref{cond:LAverageLipCond} reduces to
\begin{equation}
\label{cond:GammaCond_weak}
\|F'(x_0)^{-1}[F'(y) - F'(x)]\| \leq \frac{1}{(1-\gamma\|x-x_0\|-\|y-x\|)^2} - \frac{1}{(1-\gamma\|x-x_0\|)^2},
\end{equation}
The majorizing function $h$ defined by \eqref{fun:h(t)} reduces to
\begin{equation}
\label{fun:h(t)_GammaCond}
h(t) = \beta - t + \frac{\gamma t^2}{1 - \gamma t}, \quad t \in [0, \frac{1}{\gamma}).
\end{equation}
The constants $r_0$ and $b$ defined in \eqref{cons:r0} and \eqref{cons:b} are given by
$$
r_0 = \left(1 - \frac{1}{\sqrt{2}}\right) \frac{1}{\gamma} \quad \text{and} \quad  b = (3 - 2\sqrt{2})\frac{1}{\gamma},
$$
respectively. If $\alpha := \beta \gamma \leq 3 - 2\sqrt{2}$, then the zeros of $h$ are
\begin{equation}
\label{cons:t*t**_gammaCond}
t^* = \frac{1+\alpha - \sqrt{(1+\alpha)^2 - 8\alpha}}{4\gamma}
    \quad \text{and} \quad
t^{**} = \frac{1+\alpha + \sqrt{(1+\alpha)^2 - 8\alpha}}{4\gamma},
\end{equation}
respectively. Then, the constant $H_* := h''(t^*)/h'(t^*)$ given in Theorem \ref{th:SemilocalConv}
now has the following concrete form:
\begin{equation}
\label{cons:H*}
H_*  = - \frac{32\gamma}{\sqrt{(1+\alpha)^2 - 8\alpha} \cdot (3 - \alpha + \sqrt{(1+\alpha)^2 - 8\alpha})^2}.
\end{equation}
Consequently, we have the following Smale-type convergence result from Theorem \ref{th:SemilocalConv} for
two-step Newton method \eqref{it:TwoStepNM} under the condition \eqref{cond:GammaCond_weak}.

\begin{mycorollary}
\label{cor:Conv_GammaCond_weak}
Let $F: \DS \subset \XS \to \YS$ be a continuously Fr\'{e}chet differentiable nonlinear operator,
$\DS$ open and convex. Assume that there exists an initial point $x_0 \in \DS$ such that
$F'(x_0)^{-1}$ exists and that $F'$ satisfies the condition \eqref{cond:GammaCond_weak}.
Let $\{x_k\}$ be the sequence generated by the two-step Newton method \eqref{it:TwoStepNM}
with initial point $x_0$. If $0 < \alpha \leq 3 - 2\sqrt{2}$, then $\{x_k\}$ is well-defined and converges
Q-superquadratically to a solution
$x^* \in \overline{\ball(x_0, t^*)}$ of \eqref{eq:F(x)=0}, and this solution $x^*$
is unique in $\overline{\ball(x_0, r)}$, where $t^* \leq r < t^{**}$,
$t^*$ and $t^{**}$ are given in \eqref{cons:t*t**_gammaCond}.
Moreover, if $0 < \alpha < 3 - \sqrt[3]{2} - \sqrt[3]{4}$, then the order of convergence is cubic at least and
we have the following error bounds
\begin{equation}
\label{estimate:norm_x*-xk+1_GammaCond}
\|x^* - x_{k+1}\| \leq \frac{q H_*^2}{2}\cdot \|x^* - x_k\|^3, \quad k \geq 0,
\end{equation}
where $H_*$ is given in \eqref{cons:H*} and
$$
q := \frac{\sqrt{(1+\alpha)^2 - 8\alpha} \cdot (3 - \alpha + \sqrt{(1+\alpha)^2 - 8\alpha})^2 + 4(1 + \alpha - \sqrt{(1+\alpha)^2 - 8\alpha})}{\sqrt{(1+\alpha)^2 - 8\alpha} \cdot (3 - \alpha + \sqrt{(1+\alpha)^2 - 8\alpha})^2 - 4(1 + \alpha - \sqrt{(1+\alpha)^2 - 8\alpha})}.
$$
\end{mycorollary}

\begin{myremark}
If $F$ is twice continuously Fr\'{e}chet differentiable, then $F'$ satisfies the condition \eqref{cond:GammaCond_weak}
if and only if $F$ satisfies the following condition
\begin{equation}
\label{cond:GammaCond}
\|F'(x_0)^{-1}F''(x)\| \leq \frac{2\gamma}{(1 - \gamma \|x - x_0\|)^3}, \quad x \in \ball(x_0, 1/\gamma).
\end{equation}
In fact, if $F$ satisfies \eqref{cond:GammaCond_weak}, then \eqref{cond:GammaCond} holds trivially.
Conversely, if $F$ satisfies \eqref{cond:GammaCond}, then by noting that $h'(t) = -2 + 1/(1-\gamma t)^2$ and
$h''(t) = 2\gamma/(1-\gamma t)^3$, we have
\begin{align*}
\|F'(x_0)^{-1}[F'(y) - F'(x)]\|
    &\leq \int_0^1 \|F'(x_0)^{-1}F''(x + \tau(y-x))\|\|y-x\| \dif{\tau} \\
    &\leq \int_0^1 h''(\|x-x_0\| + \tau\|y - x\|) \|y - x\| \dif{\tau} \\
    &= h'(\|x - x_0\| + \|y - x\|) - h'(\|x - x_0\|),
\end{align*}
which means that $F$ satisfies \eqref{cond:GammaCond_weak}.
The condition \eqref{cond:GammaCond} is called the $\gamma$-condition which
was introduced by Wang and Han in \cite{WangH1997a} to study the Smale point estimate theory.
Based on the above
observation, the convergence result stated in Corollary \ref{cor:Conv_GammaCond_weak} also hold when
the condition \eqref{cond:GammaCond_weak} is replaced by the $\gamma$-condition \eqref{cond:GammaCond}.
\end{myremark}

One important and typical class of examples satisfying the $\gamma$-condition
is the one of analytic operators. Smale \cite{Smale1986} studied the convergence and
error estimate of Newton's method \eqref{it:NM} under the hypotheses that
$F$ is analytic and satisfies
$$
\|F'(x_0)^{-1}F^{(n)}(x_0)\| \leq n! \gamma^{n-1}, \quad n \geq 2,
$$
where $\gamma$ is given by
\begin{equation}
\label{cons:gamma}
\gamma := \sup_{k > 1} \left\|\frac{F'(x_0)^{-1}F^{(k)}(x_0)}{k!}\right\|^{\frac{1}{k-1}}.
\end{equation}
We then obtain from Theorem \ref{th:SemilocalConv} that another Smale-type convergence result
of the two-step Newton method \eqref{it:NM} for the analytic operator.

\begin{mycorollary}
\label{cor:Conv_AnalyticOperator}
Let $F: \DS \subset \XS \to \YS$ be an analytic operator,
$\DS$ open and convex. Assume that there exists an initial point $x_0 \in \DS$ such that
$F'(x_0)$ is nonsingular.
Let $\{x_k\}$ be the sequence generated by the two-step Newton method \eqref{it:TwoStepNM}
with initial point $x_0$. If $0 < \alpha :=\beta \gamma \leq 3 - 2\sqrt{2}$, where $\gamma$ is given
by \eqref{cons:gamma},
then $\{x_k\}$ is well-defined and converges Q-superquadratically to a solution
$x^* \in \overline{\ball(x_0, t^*)}$ of \eqref{eq:F(x)=0}, and this solution $x^*$
is unique in $\overline{\ball(x_0, r)}$, where $t^* \leq r < t^{**}$,
$t^*$ and $t^{**}$ are given in \eqref{cons:t*t**_gammaCond}.
Moreover, the order of convergence is cubic at least and
the error estimate \eqref{estimate:norm_x*-xk+1_GammaCond} holds
when $0 < \alpha < 3 - \sqrt[3]{2} - \sqrt[3]{4}$.
\end{mycorollary}

Lastly, we present a semilocal convergence result from Theorem \ref{th:SemilocalConv}
for the two-step Newton method \eqref{it:TwoStepNM}
under the condition introduced by Nesterov and Nemirovskii in the seminal work \cite{NesterovN1994}.

Let $f: \DS \subset \RS^n \to \RS$ be a strictly convex and three times continuously differentiable function,
$\DS$ open and convex. Let $\vec{x}_0 \in \DS$ be an initial point such that the inverse $f''(\vec{x}_0)^{-1}$ exists.
For a given constant $a > 0$, if $f$ satisfies the inequality
$$
|f'''(\vec{x})[\vec{u},\vec{u},\vec{u}]| \leq 2 a^{-1/2} \left(f''(\vec{x})[\vec{u},\vec{u}]\right)^{3/2},  \quad \vec{x} \in \DS, \vec{u} \in \RS^n,
$$
then we say that $f$ is $a$-self-concordant \cite{NesterovN1994}. For $\vec{x} \in \DS$, we set
$$
\left<\vec{u}_1, \vec{u}_2\right>_{\vec{x}} := a^{-1}\left<f''(\vec{x})\vec{u}_1,\vec{u}_2\right>, \quad \forall\, \vec{u}_1, \vec{u}_2 \in \RS^n,
$$
and define the norm
$$
\|\vec{u}\|_{\vec{x}} := \sqrt{\left<\vec{u}, \vec{u}\right>_{\vec{x}}}, \quad \forall\, \vec{u} \in \RS^n.
$$
Then $(\RS^n, \|\cdot\|_{\vec{x}})$ is a Banach space.

Choose $L(u) = 2/(1-u)^3$ in \eqref{fun:h(t)}. We obtain the following majorizing function
$$
h(t) = \beta - t + \frac{t^2}{1 - t}, \quad t \in [0,1),
$$
which is a special case of the one given in \eqref{fun:h(t)_GammaCond}; i.e., $\gamma \equiv 1$.
Then there exists two zeros for this majorizing function when $\beta \leq 3 - 2\sqrt{2}$. These two zeros are as follows:
\begin{equation}
\label{cons:t*t**_Nesterov}
t^* = \frac{1+\beta - \sqrt{(1+\beta)^2 - 8\beta}}{4}
    \quad \text{and} \quad
t^{**} = \frac{1+\beta + \sqrt{(1+\beta)^2 - 8\beta}}{4}.
\end{equation}
For a given vector $\vec{x} \in \RS^n$ and a positive number $r$, we set
$$
\ball_{r}(\vec{x}) := \{\vec{y} \in \RS^n: \|\vec{y} - \vec{x}\|_{\vec{x}} < r\}
    \quad \text{and} \quad
\overline{\ball_{r}(\vec{x})} := \{\vec{y} \in \RS^n: \|\vec{y} - \vec{x}\|_{\vec{x}} \leq r\}.
$$
They correspond to the open ball and its closure of center $\vec{x}$ and radius $r$ when $\RS^n$
is endowed with the metric structure induced by the preceding inner product $\left<\cdot,\cdot\right>_{\vec{x}}$.
If $f$ is an $a$-self-concordant function, then we have (see \cite[Lemma 5.1]{AlvarezBM2008})
$$
\|f''(\vec{x}_0)^{-1}f'''(\vec{x})\|_{\vec{x}_0} \leq \frac{2}{(1 - \|\vec{x} - \vec{x}_0\|_{\vec{x}_0})^3}, \quad \vec{x} \in \ball_{1}(\vec{x}_0).
$$
This implies that $f''$ satisfies the $L$-average Lipschitz condition \eqref{cond:LAverageLipCond} with $L(u) = 2/(1-u)^3$.
Let $\XS = \YS = (\RS^n, \|\cdot\|_{\vec{x}})$. Then, by Theorem \ref{th:SemilocalConv},
we have the following semilocal convergence result about the minimization of $a$-self-concordant function
for two-step Newton method which defined by
\begin{equation}
\label{it:TwoStepNM_SelfConc}
\left\{
\begin{aligned}
& \vec{y}_k = \vec{x}_k - f''(\vec{x}_k)^{-1}f'(\vec{x}_k), \\
& \vec{x}_{k+1} = \vec{y}_k - f''(\vec{x}_k)^{-1}f'(\vec{y}_k),
\end{aligned}\right.
\quad k = 0,1,2,\ldots.
\end{equation}
\begin{mycorollary}
Let $f: \DS \subset \RS^n \to \RS$ be an $a$-self-concordant function, $\DS$ open and convex.
Assume that there exists an initial point $\vec{x}_0 \in \DS$ such that $f''(\vec{x}_0)$ is nonsingluar.
Let $\{\vec{x}_k\}$ be the vector sequence generated by the two-step Newton method \eqref{it:TwoStepNM_SelfConc}
for solving $f'(\vec{x}) = \vec{0}$ with initial point $\vec{x}_0$.
If $\beta := \|f''(\vec{x}_0)f'(\vec{x}_0)\|_{\vec{x}_0} \leq 3 - 2\sqrt{2}$,
then $\{\vec{x}_k\}$ is well-defined and converges Q-superquadratically to a point
$\vec{x}^*$ which is the minimizer of $f$ in $\overline{\ball_{t^*}(\vec{x}_0)}$,
and this minimizer $\vec{x}^*$
is unique in $\overline{\ball_r(\vec{x}_0)}$, where $t^* \leq r < t^{**}$,
$t^*$ and $t^{**}$ are given in \eqref{cons:t*t**_Nesterov}.
Moreover,  if $0 < \beta < 3 - \sqrt[3]{2} - \sqrt[3]{4}$, then
the order of convergence is cubic at least and we have the following error bounds
$$
\|\vec{x}^* - \vec{x}_{k+1}\| \leq \frac{q H_*^2}{2}\cdot \|\vec{x}^* - \vec{x}_k\|^3, \quad k \geq 0,
$$
where
$$
H_*  = - \frac{32}{\sqrt{(1+\beta)^2 - 8\beta} \cdot (3 - \beta + \sqrt{(1+beta)^2 - 8\beta})^2}
$$
and
$$
q = \frac{\sqrt{(1+\beta)^2 - 8\beta} \cdot (3 - \beta + \sqrt{(1+\beta)^2 - 8\beta})^2 + 4(1 + \beta - \sqrt{(1+\beta)^2 - 8\beta})}{\sqrt{(1+\beta)^2 - 8\beta} \cdot (3 - \beta + \sqrt{(1+\beta)^2 - 8\beta})^2 - 4(1 + \beta - \sqrt{(1+\beta)^2 - 8\beta})}.
$$
\end{mycorollary}

\begin{myremark}
Semilocal convergence result on the analysis of self-concordant minimization for Newton's method \eqref{it:NM}
has already been presented by Alvarez et al. in \cite{AlvarezBM2008}.
In addition, Ferreira and Svaiter \cite{FerreiraS2012} provided another semilocal convergence result
on self-concordant minimization for Newton's method with a relative error tolerance.
\end{myremark}

\section{The proof for Theorem \ref{th:SemilocalConv}}

This section is devoted to the proof of Theorem \ref{th:SemilocalConv}. We begin with some
technical lemmas about error estimates for the majorizing sequences $\{s_k\}$ and $\{t_k\}$ defined by \eqref{it:TwoStepNM_Scalar},
and about the relationship between the majorizing function $h(t)$ defined by \eqref{fun:h(t)} and the nonlinear operator $F$.
Then, we provide the convergence analysis of the two-step Newton method \eqref{it:TwoStepNM} presented in Theorem \ref{th:SemilocalConv}.
\subsection{Technical lemmas}

Recall that the majorizing function $h$ is defined by \eqref{fun:h(t)},
the majorizing sequences $\{s_k\}$ and $\{t_k\}$ defined by \eqref{it:TwoStepNM_Scalar},
$t^*$ the zero of $h$ on $[0, r_0]$, where $r_0$ satisfies \eqref{cons:r0}.
By Lemma \ref{lem:TwoStepNM_Scalar_Conv}, $\{s_k\}$ and $\{t_k\}$
converge increasingly to $t^*$ when $0 < \beta \leq b$,
where $b$ is defined by \eqref{cons:b}.

\begin{mylemma}
Let $\{s_k\}$ and $\{t_k\}$ be the sequences generated by \eqref{it:TwoStepNM_Scalar}.
Suppose that $0 < \beta \leq b$ and $2 + t^*h''(t^*)/h'(t^*) \geq 0$. Then we have
\begin{equation}
\label{ineq:sk-tk}
s_k - t_k \geq (t^* - t_k) + \frac{h''(t^*)}{2 h'(t^*)}(t^* - t_k)^2, \quad k \geq 0.
\end{equation}
\end{mylemma}

\begin{proof}
Thanks to the definition of $\{s_k\}$ and $\{t_k\}$, we can derive
\begin{align*}
s_k - t_k
    &= (t^* - t_k) + \frac{1}{h'(t_k)} [h(t^*) - h(t_k) - h'(t_k)(t^* - t_k)] \\
    &= (t^* - t_k) + \frac{1}{h'(t_k)} \int_0^1 [h'(t_k+\tau(t^*-t_k)) - h'(t_k)](t^* - t_k) \dif{\tau}.
\end{align*}
In view of $h'$ is increasing in $[0,r_0)$, one has $h'(t_k) \leq h'(t_k + \tau(t^* - t_k)) \leq h'(t^*)$ for any $\tau \in [0,1]$.
Since $h'$ is convex in $[0, r_0)$ and satisfies $-1 \leq h'(t) < 0$ for any $t \in [0, r_0)$,
it follows from Lemma \ref{lem:ConvexProperty} that
\begin{align*}
s_k - t_k
    &\geq (t^* - t_k) + \frac{1}{h'(t^*)} \int_0^1 \frac{h'(t^*) - h'(t_k)}{t^* - t_k} \tau \dif{\tau} \cdot (t^* - t_k)^2 \\
    &\geq (t^* - t_k) + \frac{h''(t^*)}{2 h'(t^*)}(t^* - t_k)^2,
\end{align*}
which verifies the inequality in the lemma.
\end{proof}

\begin{mylemma}
Let $\{s_k\}$ and $\{t_k\}$ be the sequences generated by \eqref{it:TwoStepNM_Scalar}.
Suppose that $0 < \beta \leq b$. Then the sequence $\{t_k\}$ converges Q-cubic to $t^*$ as follows:
\begin{equation}
\label{ineq:t*-tk+1}
t^* - t_{k+1} \leq \frac{1}{2} \left(\frac{h''(t^*)}{h'(t^*)}\right)^2 (t^* - t_k)^3, \quad k \geq 0.
\end{equation}
\end{mylemma}

\begin{proof}
By \eqref{it:TwoStepNM_Scalar}, we may derive the following relation
\begin{align*}
t^* - t_{k+1}
    &= t^* - s_k + \frac{h(s_k)}{h'(t_k)} \\
    &= -  \frac{1}{h'(t_k)} \cdot [h(t^*) - h(s_k) - h'(t_k)(t^* - s_k)] \\
    &= -  \frac{1}{h'(t_k)} \int_0^1 [h'(s_k + \tau(t^* - s_k)) - h'(t_k)](t^* - s_k) \dif{\tau}.
\end{align*}
We then will apply the preceding relation to obtain the estimate \eqref{ineq:t*-tk+1}.
Using  \eqref{it:TwoStepNM_Scalar} again, we deduce that
\begin{align*}
t^* - s_k
    &= -  \frac{1}{h'(t_k)} \int_0^1 [h'(t_k + \tau(t^* - t_k)) - h'(t_k)](t^* - t_k) \dif{\tau}.
\end{align*}
Taking into account the convexity of $h'$ in $[0,r_0)$, it follows from Lemma \ref{lem:ConvexProperty} that, for any $\tau \in (0,1]$,
\begin{align*}
h'(t_k + \tau(t^* - t_k)) - h'(t_k)
    &= \frac{h'(t_k + \tau(t^* - t_k)) - h'(t_k)}{\tau(t^* - t_k)} \cdot \tau (t^* - t_k) \\
    &\leq \frac{h'(t^*) - h'(t_k)}{t^* - t_k} \cdot \tau (t^* - t_k).
\end{align*}
Then, in view of the positivity of $-1/h'(t)$, one has from Lemma \ref{lem:ConvexProperty} again that
\begin{equation}
\label{ineq:t*-sk}
t^* - s_k
    \leq -  \frac{1}{h'(t_k)} \int_0^1 \frac{h'(t^*) - h'(t_k)}{t^* - t_k} \tau \dif{\tau} \cdot (t^* - t_k)^2
    \leq - \frac{1}{2} \frac{h''(t^*)}{h'(t^*)} \cdot (t^* - t_k)^2.
\end{equation}
the last due to $h'$ is strictly increasing. On the other hand, thanks to Lemma \ref{lem:ConvexProperty} again
and note that $s_k - t_k + \tau(t^* - s_k) \leq t^* - t_k$ holds for any $\tau \in [0,1]$, we have
\begin{align*}
h'(s_k + \tau(t^* - s_k)) - h'(t_k)
    &\leq \frac{h'(t^*) - h'(t_k)}{t^* - t_k} \cdot [s_k - t_k + \tau(t^* - s_k)] \\
    &\leq h''(t^*)(t^* - t_k).
\end{align*}
This together with \eqref{ineq:t*-sk} allows us to conclude that
$$
t^* - t_{k+1}
    \leq - \frac{h''(t^*)}{h'(t^*)} (t^* - t_k) (t^* - s_k)
    \leq \frac{1}{2} \left(\frac{h''(t^*)}{h'(t^*)}\right)^2 (t^* - t_k)^3,
$$
and the bound claimed in the lemma follows.
\end{proof}

The following lemmas, which provide clear relationships between the majorizing function
and the nonlinear operator, will play key roles for the semilocal convergence analysis
of the two-step Newton method \eqref{it:TwoStepNM}.

\begin{mylemma}
\label{lem:F'(x)-1_invertibility}
Suppose $\|x - x_0\| \leq t < t^*$. If $F'$ satisfies the $L$-average Lipschitz condition \eqref{cond:LAverageLipCond}
on $\ball(x^*, t)$, then $F'(x)$ is nonsingular and
\begin{equation}
\label{ineq:norm_F'(x)-1F'(x0)}
\|F'(x)^{-1}F'(x_0)\| \leq - \frac{1}{h'(\|x - x_0\|)} \leq - \frac{1}{h'(t)}.
\end{equation}
In particular, $F'$ is nonsingular in $\ball(x_0, t^*)$.
\end{mylemma}

\begin{proof}
Take $x \in \overline{\ball(x_0, t)}, 0 \leq t < t^*$. By using the $L$-average Lipschitz condition \eqref{cond:LAverageLipCond},
we have
\begin{align*}
\|F'(x_0)^{-1} F'(x) - \I\|
     \leq \int_0^{\|x - x_0\|} L(u) \dif{u}
     = h'(\|x - x_0\|) - h'(0).
\end{align*}
Since $h'(0) = -1$ and $h'$ is strictly increasing in $(0,t^*)$,
we obtain
$$
\|F'(x_0)^{-1} F'(x) - \I\| \leq h'(t) + 1 < 1,
$$
the last due to $-1 < h'(t) < 0$ for any $t \in (0, t^*)$.
Therefore, the Banach lemma is applicable to conclude that $F'(x_0)^{-1} F'(x)$ is nonsingular and \eqref{ineq:norm_F'(x)-1F'(x0)}
holds. The proof is complete.
\end{proof}

\begin{mylemma}
\label{lem:AuxiliaryLemma1_Semilocal}
Let $\{s_k\}$ and $\{t_k\}$ be generated by \eqref{it:TwoStepNM_Scalar}. Assume that $F'$ satisfies
the $L$-average Lipschitz condition \eqref{cond:LAverageLipCond} on $\ball(x_0, t^*)$.
If $0 < \beta \leq b$, then the sequences $\{x_k\}$
and $\{y_k\}$ generated by the two-step Newton method \eqref{it:TwoStepNM}
with initial point $x_0$ are well-defined and contained in $\ball(x_0, t^*)$.
Moreover, for all $k = 0,1,2,\ldots$, we have
\begin{itemize}[font=\upshape]
\item[(i)]
$F'(x_k)^{-1}$ exists and $\|F'(x_k)^{-1}F'(x_0)\| \leq -1/h'(\|x_k - x_0\|) \leq -1/h'(t_k)$.
\item[(ii)]
$\|F'(x_0)^{-1}F(x_k)\| \leq h(t_k)$.
\item[(iii)]
$\|y_k - x_k\| \leq s_k - t_k$.
\item[(iv)]
$\|x_{k+1} - y_k\| \leq (t_{k+1} - s_k) \cdot \left(\cfrac{\|y_k - x_k\|}{s_k - t_k}\right)^2 \leq t_{k+1} - s_k$.
\item[(v)]
$\|x_{k+1} - x_k\| \leq t_{k+1} - t_k$.
\end{itemize}
\end{mylemma}

\begin{proof}
We reason by induction. The case $k = 0$ is true obviously for (i)-(iii).
Thus $y_0 \in \ball(x_0, t^*)$ owing to $\|y_0 - x_0\| \leq s_0 - t_0 = s_0 < t^*$. As for (iv) and (v), by \eqref{it:TwoStepNM}, we have
\begin{align*}
F(y_0) &= F(y_0) - F(x_0) - F'(x_0)(y_0 - x_0) \\
&= \int_0^1 [F'(x_0 + \tau(y_0 - x_0)) - F'(x_0)](y_0 - x_0) \dif{\tau}.
\end{align*}
Then, the $L$-average Lipschitz condition \eqref{cond:LAverageLipCond} is applicable to deduce that
\begin{align*}
\|F'(x_0)^{-1}F(y_0)\| &\leq \int_0^1 \|F'(x_0)^{-1}[F'(x_0 + \tau(y_0 - x_0)) - F'(x_0)]\|\|y_0 - x_0\| \dif{\tau} \\
&\leq \int_0^1 \left(\int_0^{\tau\|y_0 - x_0\|} L(u) \dif{u}\right) \|y_0 - x_0\| \dif{\tau} \\
&= \int_0^1 [h'(\tau\|y_0 - x_0\|) - h'(0)] \|y_0 - x_0\| \dif{\tau}.
\end{align*}
In view of $h'$ is strictly convex in $[0,r_0)$ and noting that $\|y_0 - x_0\| \leq s_0 - t_0$ by (iii),
it follows from Lemma \ref{lem:ConvexProperty} that
\begin{align*}
h'(\tau\|y_0 - x_0\|) - h'(0)
    &= \frac{h'(\tau\|y_0 - x_0\|) - h'(0)}{\|y_0 - x_0\|} \cdot \|y_0 - x_0\| \\
    &\leq \frac{h'(\tau(s_0 - t_0)) - h'(0)}{s_0 - t_0} \cdot \|y_0 - x_0\|.
\end{align*}
Then, combining the above inequality and \eqref{it:TwoStepNM_Scalar}, one has that
\begin{align*}
\|F'(x_0)^{-1}F(y_0)\|
&\leq \int_0^1 [h'(\tau s_0) - h'(0)]s_0 \dif{\tau}\cdot \left(\frac{\|y_0 - x_0\|}{s_0 - t_0}\right)^2 \\
&= h(s_0) \cdot \left(\frac{\|y_0 - x_0\|}{s_0 - t_0}\right)^2
    = (t_1 - s_0) \cdot \left(\frac{\|y_0 - x_0\|}{s_0 - t_0}\right)^2.
\end{align*}
This leads to
$$
\|x_1 - y_0\| = \|F'(x_0)^{-1}F(y_0)\| \leq (t_1 - s_0) \cdot \left(\frac{\|y_0 - x_0\|}{s_0 - t_0}\right)^2.
$$
Hence, we have
$$
\|x_1 - x_0\| \leq \|x_1 - y_0\| + \|y_0 - x_0\| \leq (t_1 - s_0) + (s_0 - t_0) = t_1 - t_0.
$$
That is to say, (iv) and (v) hold for the case $k = 0$, which implies that $x_1 \in \ball(x_0,t^*)$.
Now we assume that $x_k, y_k \in \ball(x_0,t^*)$, $\|x_k - x_0\| \leq t_k$ and (i)-(v) hold for some $k \geq 0$.
Then, applying the inductive hypothesis (iii) and Lemma \ref{lem:TwoStepNM_Scalar_Conv}, we obtain that
$\|y_k - x_0\| \leq \|y_k - x_k\| + \|x_k - x_0\| \leq s_k$.
In addition, we use the inductive hypothesis (v) and
Lemma \ref{lem:TwoStepNM_Scalar_Conv} to yield
$$
\|x_{k+1} - x_0\| \leq \sum_{i = 0}^k \|x_{i+1} - x_i\| \leq \sum_{i = 0}^k (t_{i+1} - t_i) = t_{k+1} < t^*,
$$
which implies that $x_{k+1} \in \ball(x_0, t^*)$.
This together with \eqref{ineq:norm_F'(x)-1F'(x0)} gives that (i) holds for the case $k+1$.
For (ii), by \eqref{it:TwoStepNM} again, we have the following identity:
\begin{align*}
F(x_{k+1}) &= F(x_{k+1}) - F(y_k) - F'(x_k)(x_{k+1} - y_k) \\
&= \int_0^1 [F'(y_k + \tau(x_{k+1} - y_k)) - F'(x_k)](x_{k+1} - y_k) \dif{\tau}.
\end{align*}
It follows from the $L$-average Lipschitz condition \eqref{cond:LAverageLipCond} that
\begin{align*}
\|F'(x_0)^{-1}F(x_{k+1})\|
    &\leq \int_0^1 \|F'(x_0)^{-1}[F'(y_k + \tau(x_{k+1} - y_k)) - F'(x_k)]\|\|x_{k+1} - y_k\| \dif{\tau} \\
    &\leq \int_0^1 \left(\int_{\|x_k-x_0\|}^{\|x_k - x_0\| + \|y_k - x_k + \tau(x_{k+1} - y_k)\|} L(u) \dif{u}\right) \|x_{k+1} - y_k\| \dif{\tau}.
\end{align*}
In view of $h'$ is increasing and convex in $[0,r_0)$,
by applying Lemma \ref{lem:ConvexProperty} and the inductive hypotheses (iii)-(iv), one has that
\begin{align*}
\lefteqn{\int_{\|x_k-x_0\|}^{\|x_k - x_0\| + \|y_k - x_k + \tau(x_{k+1} - y_k)\|} L(u) \dif{u}} \\
    &= h'(\|x_k - x_0\| + \|y_k - x_k + \tau(x_{k+1} - y_k)\|) - h'(\|x_k-x_0\|) \\
    &\leq h'(\|x_k - x_0\| + \|y_k - x_k\| + \tau\|x_{k+1} - y_k\|) - h'(\|x_k-x_0\|) \\
    &\leq \frac{h'(s_k + \tau(t_{k+1}-s_k)) - h'(t_k)}{s_k - t_k + \tau(s_k - t_k)}\cdot (\|y_k - x_k\| + \tau\|x_{k+1} - y_k\|) \\
    &\leq h'(s_k + \tau(t_{k+1}-s_k)) - h'(t_k).
\end{align*}
This allows us to get
\begin{align}
\|F'(x_0)^{-1}F(x_{k+1})\|
    &\leq \int_0^1[h'(s_k + \tau(t_{k+1}-s_k)) - h'(t_k)]\|x_{k+1} - y_k\| \dif{\tau} \nonumber\\
    &= [h(t_{k+1}) - h(s_k) - h'(t_k)(t_{k+1} - s_k)] \cdot \frac{\|x_{k+1} - y_k\|}{t_{k+1} - s_k} \nonumber\\
    &= h(t_{k+1}) \cdot \frac{\|x_{k+1} - y_k\|}{t_{k+1} - s_k} \nonumber\\
    &\leq h(t_{k+1}), \label{ineq:norm_F(xk+1)}
\end{align}
which shows that (ii) holds for the case $k+1$. Combining \eqref{ineq:norm_F'(x)-1F'(x0)} and \eqref{ineq:norm_F(xk+1)},
we further obtain that
\begin{align}
\|y_{k+1} - x_{k+1}\|
    &= \|F'(x_{k+1})^{-1}F(x_{k+1})\| \nonumber\\
    &\leq \|F'(x_{k+1})^{-1}F'(x_0)\|\|F'(x_0)^{-1}F(x_{k+1})\| \nonumber\\
    &\leq -\frac{h(t_{k+1})}{h'(t_{k+1})}
        = s_{k+1} - t_{k+1}. \label{ineq:norm_yk+1-xk+1}
\end{align}
This means that (iii) holds for the case $k+1$. Then, we conclude that
$\|y_{k+1} - x_0\| \leq \|y_{k+1} - x_{k+1}\| + \|x_{k+1} - x_0\| \leq s_{k+1} < t^*$ and so
$y_{k+1} \in \ball(x_0, t^*)$. As for (iv), noting that
\begin{align*}
x_{k+2} - y_{k+1}
    &= - F'(x_{k+1})^{-1} F(y_{k+1}) \\
    &= - F'(x_{k+1})^{-1}[F(y_{k+1}) - F(x_{k+1}) - F'(x_{k+1})(y_{k+1} - x_{k+1})] \\
    &= - F'(x_{k+1})^{-1} \int_0^1 [F'(x_{k+1}^\tau) - F'(x_{k+1})] (y_{k+1} - x_{k+1}) \dif{\tau},
\end{align*}
where $x_{k+1}^\tau := x_{k+1}+\tau(y_{k+1} - x_{k+1})$, by using \eqref{ineq:norm_F'(x)-1F'(x0)},
the $L$-average Lipschitz condition \eqref{cond:LAverageLipCond}, we have
\begin{align*}
\|x_{k+2} - y_{k+1}\|
    &\leq - \frac{1}{h'(t_{k+1})}
        \int_0^1 \left[\int_{\|x_{k+1} - x_0\|}^{\|x_{k+1} - x_0\|+\tau\|y_{k+1}-x_{k+1}\|} L(u) \dif{u}\right] \|y_{k+1}-x_{k+1}\| \dif{\tau}.
\end{align*}
Taking into account that $h'$ is increasing and convex in $[0,r_0)$ again,
by combining \eqref{ineq:norm_yk+1-xk+1} with Lemma \ref{lem:ConvexProperty}, one gets that
\begin{align*}
\lefteqn{\int_{\|x_{k+1} - x_0\|}^{\|x_{k+1} - x_0\|+\tau\|y_{k+1}-x_{k+1}\|} L(u) \dif{u}} \\
    &= h'(\|x_{k+1} - x_0\|+\tau\|y_{k+1}-x_{k+1}\|) - h'(\|x_{k+1} - x_0\|) \\
    &= \frac{h'(\|x_{k+1} - x_0\|+\tau\|y_{k+1}-x_{k+1}\|) - h'(\|x_{k+1} - x_0\|)}{\|y_{k+1}-x_{k+1}\|} \cdot \|y_{k+1}-x_{k+1}\| \\
    &\leq  \frac{h'(t_{k+1}+\tau(s_{k+1}-t_{k+1})) - h'(t_{k+1})}{s_{k+1}-t_{k+1}} \cdot \|y_{k+1}-x_{k+1}\|.
\end{align*}
This permits us to arrive at
\begin{align*}
\lefteqn{\|x_{k+2} - y_{k+1}\|} \\
    &\leq - \frac{1}{h'(t_{k+1})} \int_0^1 [h'(t_{k+1}+\tau(s_{k+1}-t_{k+1})) - h'(t_{k+1})] \cdot \frac{\|y_{k+1}-x_{k+1}\|^2}{s_{k+1}-t_{k+1}} \dif{\tau} \\
    &= - \frac{1}{h'(t_{k+1})}[h(s_{k+1}) - h(t_{k+1}) - h'(t_{k+1})(s_{k+1}-t_{k+1})] \cdot \left(\frac{\|y_{k+1}-x_{k+1}\|}{s_{k+1}-t_{k+1}}\right)^2 \\
    &= (t_{k+2} - s_{k+1}) \cdot \left(\frac{\|y_{k+1}-x_{k+1}\|}{s_{k+1}-t_{k+1}}\right)^2.
\end{align*}
Furthermore, we derive from this together with \eqref{ineq:norm_yk+1-xk+1} that
\begin{equation*}
\|x_{k+2} - x_{k+1}\| \leq \|x_{k+2} - y_{k+1}\| + \|y_{k+1} - x_{k+1}\| \leq t_{k+2} - t_{k+1}.
\end{equation*}
Therefore, all the statements in the lemma hold by induction. This completes the proof.
\end{proof}

\begin{mylemma}
\label{lem:AuxiliaryLemma2_Semilocal}
Under the same assumptions of Lemma $\ref{lem:AuxiliaryLemma1_Semilocal}$.
Then, the sequence $\{x_k\}$ converges to a point $x^* \in \overline{\ball(x_0, t^*)}$ with $F(x^*) = 0$.
Moreover, we have
\begin{equation}
\label{ineq:norm_x*-xk}
\|x^* - x_k\| \leq t^* - t_k, \quad k \geq 0,
\end{equation}
and
\begin{equation}
\label{ineq:norm_x*-yk}
\|x^* - y_k\| \leq (t^* - s_k) \left(\frac{\|x^*-x_k\|}{t^*-t_k}\right)^2, \quad k \geq 0.
\end{equation}
\end{mylemma}

\begin{proof}
We apply Lemma \ref{lem:AuxiliaryLemma1_Semilocal} (v)
and Lemma \ref{lem:TwoStepNM_Scalar_Conv} to obtain that
$$
\sum_{k = N}^\infty \|x_{k+1} - x_k\| \leq \sum_{k = N}^\infty (t_{k+1} - t_k) = t^* - t_N < + \infty \quad \text{for any } N \in \NS.
$$
Thus, $\{x_k\}$ is a Cauchy sequence in $\ball(x_0, t^*)$ and so converges to some $x^* \in \overline{\ball(x_0, t^*)}$.
The above inequality also implies that $\|x^* - x_k\| \leq t^* - t_k$ for any $k \geq 0$. Next, we
show that $F(x^*) = 0$. It follows from Lemma \ref{lem:F'(x)-1_invertibility} that
$\{\|F'(x_k)\|\}$ is bounded. By Lemma \ref{lem:AuxiliaryLemma1_Semilocal}, we have
$$
\|F(x_k)\| \leq \|F'(x_k)\| \|F'(x_k)^{-1} F(x_k)\| \leq \|F'(x_k)\| (s_k - t_k).
$$
Letting $k \to \infty$, by noting the fact that $\{s_k\}$ and $\{t_k\}$ are converge to the same point $t^*$ (by Lemma \ref{lem:TwoStepNM_Scalar_Conv}), we get that $\lim\limits_{k \to \infty} F(x_k) = 0$. Since $F$ is
continuous in $\overline{\ball(x_0, t^*)}$, $\{x_k\} \subset \ball(x_0, t^*)$ and $\{x_k\}$ converges to $x^*$,
we also have $\lim\limits_{k \to \infty} F(x_k) = F(x^*)$, which verifies that $F(x^*) = 0$.
It remains to show the estimate \eqref{ineq:norm_x*-yk}.
Due to Lemma \ref{lem:AuxiliaryLemma1_Semilocal}, we have
\begin{equation}
\label{ineq:norm_yk-x0}
\|y_k - x_0\| \leq \|y_k - x_k\| + \|x_k - x_0\| \leq s_k.
\end{equation}
On the other hand, we can derive the following identity:
\begin{align*}
x^* - y_k = - F'(x_k)^{-1} \int_0^1 [F'(x_k + \tau(x^* - x_k)) - F'(x_k)](x^* - x_k) \dif{\tau}.
\end{align*}
Then, in view of $h'$ is increasing and convex in $[0,r_0)$,
by combining \eqref{ineq:norm_F'(x)-1F'(x0)}, the $L$-average Lipschitz condition \eqref{cond:LAverageLipCond}
and Lemma \ref{lem:ConvexProperty}, one gets that
\begin{align*}
\|x^* - y_k\|
    &\leq - \frac{1}{h'(t_k)} \int_0^1 \left(\int_{\|x_k-x_0\|}^{\|x_k-x_0\|+\tau\|x^*-x_k\|} L(u) \dif{u}\right) \|x^*-x_k\| \dif{\tau} \\
    &= - \frac{1}{h'(t_k)} \int_0^1[h'(\|x_k-x\|+\tau\|x^*-x_k\|) - h'(\|x_k-x_0\|)] \|x^*-x_k\| \dif{\tau} \\
    &\leq - \frac{1}{h'(t_k)} \int_0^1 \frac{h'(t_k+\tau(t^*-t_k)) - h'(t_k)}{t^*-t_k} \dif{\tau} \cdot \|x^*-x_k\|^2 \\
    &= (t^* - s_k) \cdot \left(\frac{\|x^* - x_k\|}{t^* - t_k}\right)^2,
\end{align*}
as claimed. The proof of this lemma is complete.
\end{proof}

\begin{mylemma}
Under the same assumptions of Lemma $\ref{lem:AuxiliaryLemma1_Semilocal}$ and the assumption that
$2 + t^*h''(t^*)/h'(t^*) > 0$, we have
\begin{equation}
\label{ineq:norm_yk-xk_x*-xk}
\frac{\|y_k - x_k\|}{s_k - t_k}
    \leq \frac{1 - \frac{h''(t^*)}{2 h'(t^*)} (t^*-t_k)}{1 + \frac{h''(t^*)}{2 h'(t^*)} (t^*-t_k)}
        \cdot \frac{\|x^* - x_k\|}{t^* - t_k}, \quad k \geq 0.
\end{equation}
\end{mylemma}

\begin{proof}
Since $\|y_k - x_k\| \leq \|x^* - y_k\| + \|x^* - x_k\|$,
it follows from \eqref{ineq:norm_x*-xk} and \eqref{ineq:norm_x*-yk} that
\begin{align*}
\|y_k - x_k\|
    &\leq \frac{t^* - s_k}{(t^* - t_k)^2} \|x^* - x_k\|^2 + \|x^* - x_k\| \\
    &\leq \left(\frac{t^* - s_k}{t^* - t_k} + 1\right) \|x^* - x_k\|.
\end{align*}
Then, by \eqref{ineq:t*-sk}, we can get further that
$$
\|y_k - x_k\| \leq \left(1 - \frac{h''(t^*)}{2 h'(t^*)}(t^* - t_k)\right) \|x^* - x_k\|.
$$
Thus, we conclude from \eqref{ineq:sk-tk} that
\begin{align*}
\frac{\|y_k - x_k\|}{s_k - t_k}
    &\leq \frac{1 - \frac{h''(t^*)}{2 h'(t^*)}(t^* - t_k)}{s_k - t_k} \|x^* - x_k\| \\
    &\leq \frac{1 - \frac{h''(t^*)}{2 h'(t^*)} (t^*-t_k)}{1 + \frac{h''(t^*)}{2 h'(t^*)} (t^*-t_k)}
        \cdot \frac{\|x^* - x_k\|}{t^* - t_k},
\end{align*}
which yields the desired result.
\end{proof}

\subsection{Proof for Theorem \ref{th:SemilocalConv}}

Based on the technical lemmas given in previous subsection,
we are now ready to prove the semilocal convergence result given in Theorem \ref{th:SemilocalConv}
for the two-step Newton method \eqref{it:TwoStepNM}.

\begin{proof}[Proof of Theorem $\ref{th:SemilocalConv}$]
Thanks to Lemma \ref{lem:AuxiliaryLemma1_Semilocal}, we conclude that the sequence $\{x_k\}$ is well defined.
By using Lemma \ref{lem:AuxiliaryLemma1_Semilocal} (v) and Lemma \ref{lem:TwoStepNM_Scalar_Conv},
one has that $\|x_k - x_0\| \leq t_k < t^*$ for any $k \geq 0$, which means that
$\{x_k\}$ is contained in $\ball(x_0, t^*)$. Moreover, it follows from
Lemma \ref{lem:AuxiliaryLemma2_Semilocal} that $\{x_k\}$ converges to $x^*$, a solution of
\eqref{eq:F(x)=0} in $\overline{\ball(x_0, t^*)}$.
Next, we will verify the superquadratic and cubic convergence of the iterate. To do this, we
apply standard analytical techniques to derive that
\begin{align*}
\lefteqn{x^* - x_{k+1} = x^* - y_k + F'(x_k)^{-1}F(y_k)} \\
    &= - F'(x_k)^{-1}[F(x^*) - F(y_k) - F'(y_k)(x^* - y_k) + (F'(y_k) - F'(x_k))(x^* - y_k)] \\
    &= - F'(x_k)^{-1}\left[\int_0^1 \left(F'(y_k^\tau) - F'(y_k)\right)(x^* - y_k) \dif{\tau}
        + (F'(y_k) - F'(x_k))(x^* - y_k)\right],
\end{align*}
where $y_k^\tau := y_k + \tau(x^* - y_k)$. By \eqref{ineq:norm_F'(x)-1F'(x0)}
and the $L$-average Lipschitz condition \eqref{cond:LAverageLipCond}, we have
\begin{align*}
\|x^* - x_{k+1}\|
    &\leq - \frac{1}{h'(t_k)}\left[\int_0^1 \left(\int_{\|y_k - x_0\|}^{\|y_k-x_0\|+\tau\|x^*-y_k\|}
        L(u) \dif{u}\right)\|x^* - y_k\|\dif{\tau}\right. \\
    & \quad + \left.\int_{\|x_k-x_0\|}^{\|x_k-x_0\|+\|y_k-x_k\|} L(u) \dif{u} \cdot \|x^*-y_k\|\right].
\end{align*}
Taking into account $h'$ is increasing and convex in $[0,r_0)$,
combining \eqref{ineq:norm_x*-yk}, \eqref{ineq:norm_yk-x0}, Lemma \ref{lem:ConvexProperty}
and Lemma \ref{lem:AuxiliaryLemma1_Semilocal} (iii), one can deduce that
\begin{align}
\|x^* - x_{k+1}\|
    &\leq - \frac{1}{h'(t_k)}\left[\int_0^1 \frac{h'(s_k+\tau(t^*-s_k)) - h'(s_k)}{t^* - s_k} \dif{\tau}
        \cdot \|x^* - y_k\|^2 \right. \nonumber\\
    &\quad + \left.\frac{h'(s_k) - h'(t_k)}{s_k - t_k} \cdot \|y_k - x_k\| \|x^* - y_k\|\right] \nonumber\\
    &= - \frac{1}{h'(t_k)}\left[\left(h(t^*) - h(s_k) - h'(s_k)(t^*-s_k)\right)\cdot
        \left(\frac{\|x^*-y_k\|}{t^*-s_k}\right)^2 \right. \nonumber\\
    &\quad + \left. \left(h'(s_k) - h'(t_k)\right)(t^*-s_k) \cdot \frac{\|y_k-x_k\|}{s_k-t_k} \frac{\|x^*-y_k\|}{t^*-s_k}\right]. \label{ineq:norm_x*-xk+1_1}
\end{align}
By Lemma \ref{lem:AuxiliaryLemma1_Semilocal} (iii) and \eqref{ineq:norm_x*-yk} again,
the above inequality can be derive further that
\begin{equation}
\label{ineq:norm_x*-xk+1_2}
\|x^* - x_{k+1}\| \leq (t^* - t_{k+1}) \left(\frac{\|x^*-x_k\|}{t^*-t_k}\right)^2.
\end{equation}
Then, it follows from \eqref{ineq:t*-tk+1} that
$$
\frac{\|x^* - x_{k+1}\|}{\|x^*-x_k\|^2}
    \leq \frac{t^* - t_{k+1}}{(t^* - t_k)^2}
    \leq \frac{1}{2}\left(\frac{h''(t^*)}{h'(t^*)}\right)^2(t^* - t_k).
$$
Letting $k \to \infty$ in the above inequalities, by noting that $\{t_k\}$ converges to $t^*$, we have
$$
\lim_{k \to \infty} \frac{\|x^* - x_{k+1}\|}{\|x^*-x_k\|^2}  = 0,
$$
which means that $\{x_k\}$ converges Q-superquadratically to $x^*$
(See Definition \ref{def:Q_Order_Conv} for the definition).
In addition,
if the condition \eqref{cond:CubicConv_Cond} is also satisfied, then
the estimates \eqref{ineq:norm_x*-yk}, \eqref{ineq:norm_yk-xk_x*-xk} and \eqref{ineq:t*-tk+1} are
applicable to conclude from \eqref{ineq:norm_x*-xk+1_1} further that
\begin{align*}
\|x^* - x_{k+1}\|
    &\leq - \frac{1}{h'(t_k)}\left[(h(t^*) - h(s_k) - h'(s_k)(t^*-s_k)) \right. \\
    &\quad + \left. (h'(s_k) - h'(t_k))(t^*-s_k)\right]
    \cdot \frac{1 - \frac{h''(t^*)}{2 h'(t^*)} (t^*-t_k)}{1 + \frac{h''(t^*)}{2 h'(t^*)} (t^*-t_k)}
    \cdot \left(\frac{\|x^*-x_k\|}{t^*-t_k}\right)^3 \\
    &= (t^* - t_{k+1}) \cdot \frac{1 - \frac{h''(t^*)}{2 h'(t^*)} (t^*-t_k)}{1 + \frac{h''(t^*)}{2 h'(t^*)} (t^*-t_k)} \left(\frac{\|x^*-x_k\|}{t^*-t_k}\right)^3 \\
    &\leq \frac{1}{2}\left(\frac{h''(t^*)}{h'(t^*)}\right)^2
        \cdot \frac{1 - \frac{t^* h''(t^*)}{2 h'(t^*)}}{1 + \frac{t^* h''(t^*)}{2 h'(t^*)}}
        \cdot \|x^* - x_k\|^3.
\end{align*}
This shows the estimate \eqref{estimate:norm_x*-xk+1_LAverageLipCond} in Theorem \ref{th:SemilocalConv}
and so the order of convergence for the iterate is Q-cubic.

Finally, we show the uniqueness of the solution. We first to show the solution $x^*$ of \eqref{eq:F(x)=0}
is unique on $\overline{\ball(x_0, t^*)}$. Assume that there exists another solution
$x^{**}$ on $\overline{\ball(x_0, t^*)}$. Then $\|x^{**} - x_0\| \leq t^*$. Now we prove by induction that
\begin{equation}
\label{ineq:norm_x**-xk}
\|x^{**} - x_k\| \leq t^* - t_k, \quad k = 0,1,2,\ldots.
\end{equation}
It is clear that the case $k = 0$ holds because of $t_0 = 0$. Assume that the above inequality holds for some
$k \geq 0$. As the same process on the estimate $\|x^* - y_k\|$ in \eqref{ineq:norm_x*-yk}, we get
$$
\|x^{**} - y_k\| \leq (t^* - s_k)\left(\frac{\|x^{**} - x_k\| }{t^* - t_k}\right)^2.
$$
In addition, following the same process on the estimate $\|x^* - x_{k+1}\|$ in \eqref{ineq:norm_x*-xk+1_2},
we have
$$
\|x^{**} - x_{k+1}\| \leq (t^* - t_{k+1})\cdot \left(\frac{\|x^{**} - x_k\|}{t^* - t_k}\right)^2.
$$
Then, by applying the inductive hypothesis \eqref{ineq:norm_x**-xk} to the above inequality, one has that
\eqref{ineq:norm_x**-xk} also holds for the case $k+1$. Since $\{x_k\}$ converges to $x^*$
and $\{t_k\}$ converges to $t^*$, we conclude from \eqref{ineq:norm_x**-xk} that $x^{**} = x^*$.
Therefore, $x^*$ is the unique zero of \eqref{eq:F(x)=0} on $\overline{\ball(x_0, t^*)}$.
It remains to prove that $F$ does not have zeros in $\ball(x_0, r)\backslash\overline{\ball(x_0, t^*)}$.
For proving this fact by contradiction, assume that $F$ does have a zero there, that is, there exists
$x^{**} \in \DS \subset \XS$ such that $t^* < \|x^{**} - x_0\| < r$ and $F(x^{**}) = 0$.
we will show that the preceding assumptions do not hold. Firstly, we have the following observation
\begin{equation}
\label{eq:F(x**)}
F(x^{**}) = F(x_0) + F'(x_0)(x^{**} - x_0) + \int_0^1 [F'(x_0^\tau) - F'(x_0)](x^{**} - x_0) \dif{\tau},
\end{equation}
where $x_0^\tau := x_0+\tau(x^{**}-x_0)$. Note that,
\begin{align*}
\|F'(x_0)^{-1}[F(x_0) + F'(x_0)(x^{**} - x_0)]\|
    &\geq \|x^{**} - x_0\| - \|F'(x_0)^{-1}F(x_0)\| \\
    &= \|x^{**} - x_0\| - h(0).
\end{align*}
In addition, we use the $L$-average Lipschitz condition \eqref{cond:LAverageLipCond} to yield
\begin{align*}
\lefteqn{\left\|F'(x_0)^{-1}\int_0^1 [F'(x_0^\tau) - F'(x_0)](x^{**} - x_0) \dif{\tau}\right\|} \\
    &\leq \int_0^1 \left(\int_0^{\tau\|x^{**}-x_0\|} L(u) \dif{u}\right) \|x^{**}-x_0\| \dif{u}  \\
    &= \int_0^1[h'(\tau\|x^{**}-x_0\|) - h'(0)]\|x^{**}-x_0\| \dif{\tau}  \\
    &= h(\|x^{**}-x_0\|) - h(0) - h'(0) \cdot \|x^{**}-x_0\|.
\end{align*}
In view of $F(x^{**}) = 0$ and $h'(0) = -1$, we obtain from \eqref{eq:F(x**)} that
$$
h(\|x^{**}-x_0\|) - h(0) - h'(0) \cdot \|x^{**}-x_0\| \geq \|x^{**} - x_0\| - h(0),
$$
which is equivalent to $h(\|x^{**}-x_0\|) \geq 0$. It follows from Lemma \ref{lem:h(t)_property} that
$h$ is strictly positive in the interval $(\|x^{**}-x_0\|, R)$. Thus, we know $r \leq \|x^{**}-x_0\|$,
which is a contradiction to the preceding assumptions. Therefore, $F$ does not have zeros in
$\ball(x_0, r)\backslash\overline{\ball(x_0, t^*)}$ and $x^*$ is the unique zero of equation \eqref{eq:F(x)=0}
in $\ball(x_0, r)$. The proof is complete.
\end{proof}

\section{Application to algebraic Riccati equation}

In this section, we apply the two step Newton method \eqref{it:TwoStepNM} to solve
a special nonlinear vector equation which is obtained by a nonsymmetric algebraic Riccati equation (NSARE)
arising from transport theory.
Throughout this section, we use the following definitions and notations.
We call matrix $A = (a_{ij})_{m\times n} \in \RS^{m\times n}$ a positive matrix (nonnegative matrix)
if $a_{ij} > 0 \, (a_{ij} \geq 0)$ hold for all $i = 1, 2, \ldots, m, j = 1,2\ldots, n$.
If all the components of a vector are positive (negative), we call it a positive (negative) vector.
For a given a vector $\vec{a}$, we denote by $\diag(\vec{a})$ the
diagonal matrix whose diagonal elements are the components of $\vec{a}$.
We denote the vectors of all zeros and ones with proper dimension by $\vec{0}$ and $\vec{e}$, respectively.
The norm of a vector or a matrix used in this section is $\infty-$norm.

The form of the NSARE is as follows:
\begin{equation}
\label{eq:NSARE}
XCX - XD - AX + B = 0,
\end{equation}
where $A, B, C, D \in \RS^{n \times n}$ are known matrices given by
\begin{equation}
\label{mat:ABCD}
A = \Delta - \vec{e} \tran{\vec{q}}, \quad
B = \vec{e}\tran{\vec{e}}, \quad
C = \vec{q}\tran{\vec{q}}, \quad
D = \Gamma - \vec{q}\tran{\vec{e}},
\end{equation}
with
\begin{equation*}
\left\{
\begin{aligned}
\Delta &= \diag(\delta_1, \delta_2, \ldots, \delta_n)
    \text{ with }
    \delta_i = \frac{1}{c\omega_i(1 + \alpha)} > 0, \\
\Gamma &= \diag(\gamma_1, \gamma_2, \ldots, \gamma_n)
    \text{ with }
    \gamma_i = \frac{1}{c\omega_i(1 - \alpha)} > 0, \\
\vec{q} &= \tran{(q_1, q_2, \ldots, q_n)}
    \text{ with }
    q_i = \frac{c_i}{2\omega_i} > 0.
\end{aligned}
\right.
\end{equation*}
Here $c \in (0,1]$ and $\alpha \in [0,1)$. Moreover, $\{\omega_i\}_{i=1}^n$ and $\{c_i\}_{i=1}^n$ are the sets of
the Gauss-Legendre nodes and weights, respectively, on the interval
$[0,1]$, and satisfy
$$
0 < \omega_n < \cdots < \omega_2 < \omega_1 < 1
    \mbox{ and }
\sum_{i=1}^n c_i = 1, c_i > 0, i =1, 2, \ldots, n.
$$
The NSARE \eqref{eq:NSARE} has positive solutions (that is, the solution is a positive matrix), but only the minimal positive solution
of it is physically meaningful \cite{Juang1995}.

Lu \cite{Lu2005a} first proved that the solution of (\ref{eq:NSARE}) must have the following form:
$$
X = T \circ (\vec{u} \tran{\vec{v}}) = (\vec{u} \tran{\vec{v}}) \circ T,
$$
where $\circ$ denotes the Hadamard product,
$T = (t_{ij})_{n \times n} = \left(\frac{1}{\delta_i + \gamma_j}\right)_{n\times n}$,
$\vec{u}$ and $\vec{v}$ are vectors satisfying
\begin{equation}
\label{eq:VecEq_uv}
\left\{
\begin{aligned}
\vec{u} &= \vec{u} \circ (P\vec{v}) + \vec{e},\\
\vec{v} &= \vec{v} \circ (P\vec{u}) + \vec{e},
\end{aligned}
\right.
\end{equation}
with
\begin{equation}
\label{mat:P_Ptilde}
P
    = (p_{ij})_{n\times n}
    = \left(\frac{q_j}{\delta_i + \gamma_j}\right)_{n\times n}, \quad \tilde{P}
    = (\tilde{p}_{ij})_{n\times n}
    = \left(\frac{q_j}{\gamma_i + \delta_j}\right)_{n\times n}.
\end{equation}
Define nonlinear operator $\vf: \RS^{2n} \to \RS^{2n}$ by
\begin{equation}
\label{eq:f(u,v)=0}
\vf(\vec{u},\vec{v}) =
\begin{bmatrix} \vec{u} \\ \vec{v} \end{bmatrix} -
\begin{bmatrix} \vec{u} \\ \vec{v} \end{bmatrix} \circ \begin{bmatrix} P\vec{v} \\ \tilde{P}\vec{u} \end{bmatrix} -
\begin{bmatrix} \vec{e} \\ \vec{e} \end{bmatrix}.
\end{equation}
Then, one can rewrite (\ref{eq:VecEq_uv}) as $\vf(\vec{u},\vec{v}) = \vec{0}$.
Hence, the minimal positive solution of NSARE (\ref{eq:NSARE}) can be obtained via computing the
minimal positive solution of the nonlinear vector equation (\ref{eq:f(u,v)=0}). There
have been a lot of studies about the monotone convergence of various iterative methods
for solving the minimal positive solution of (\ref{eq:NSARE}),
one can see \cite{Lu2005a,BaiGL2008,LingX2017} and references therein.

Clearly, $\vec{f}$ is a continuously Fr\'{e}chet differentiable nonlinear operator in $\RS^{2n}$.
The Jacobian matrix of $\vec{f}$ at point
$(\vec{u},\vec{v})$ has the following form:
\begin{equation}
\label{eq:JacMat_f'(u,v)}
\vec{f}'(\vec{u},\vec{v}) = I_{2n} - G(\vec{u},\vec{v}),
\end{equation}
with
\begin{equation}
\label{mat:G(u,v)}
G(\vec{u},\vec{v})
    = \begin{bmatrix}
        G_1(\vec{v}) & H_1(\vec{u}) \\
        H_2(\vec{v}) & G_2(\vec{u})
        \end{bmatrix},
\end{equation}
where $I_{2n}$ is the identity matrix of order $2n$,
\begin{equation*}
\left\{
\begin{aligned}
G_1(\vec{v}) & = \diag(P\vec{v}), \\
G_2(\vec{u}) & = \diag(\tilde{P}\vec{u}), \\
H_1(\vec{u}) & = [\vec{u} \circ \vec{p}_1, \vec{u} \circ \vec{p}_2, \ldots, \vec{u} \circ \vec{p}_n], \\
H_2(\vec{v}) & = [\vec{v} \circ \tilde{\vec{p}}_1, \vec{v} \circ \tilde{\vec{p}}_2, \ldots, \vec{v} \circ \tilde{\vec{p}}_n],
\end{aligned}
\right.
\end{equation*}
$\vec{p}_i$ and $\tilde{\vec{p}}_i$ are the $i$th column of $P$ and $\tilde{P}$ for each $i = 1,2,\ldots,n$, respectively.
Choose initial point $\tran{[\tran{\vec{u}}_0,\tran{\vec{v}}_0]} = \tran{\vec{0}}$.
Then we have $\vec{f}(\vec{u}_0, \vec{v}_0) = - \vec{e}$ and $\vec{f}'(\vec{u}_0, \vec{v}_0) = I_{2n}$.
Thus,
$$
\beta := \|\vf'(\vec{u}_0, \vec{v}_0)^{-1} \vf(\vec{u}_0, \vec{v}_0)\|_{\infty}
    = \|\vec{e}\|_{\infty} = 1.
$$
Moreover, for any $(\vec{u},\vec{v})$ and $(\vec{u}',\vec{v}') \in \RS^{2n}$, it follows from \eqref{eq:JacMat_f'(u,v)} that
\begin{align*}
\|\vf'(\vec{u}_0, \vec{v}_0)^{-1}[\vf'(\vec{u},\vec{v}) - \vf'(\vec{u}',\vec{v}')]\|_{\infty}
    &= \|G(\vec{u},\vec{v}) - G(\vec{u}',\vec{v}')\|_{\infty} \\
    &\leq 2 \max_{1\leq i \leq n} \left\{\sum_{j=1}^n p_{ij}, \sum_{j=1}^n \tilde{p}_{ij}\right\}
        \cdot \left\|\begin{bmatrix} \vec{u} - \vec{u}' \\ \vec{v} - \vec{v}' \end{bmatrix}\right\|_{\infty}.
\end{align*}
In \cite[Lemma 3]{Lu2005a}, Lu derived that $\sum\limits_{j=1}^n p_{ij} < \frac{c(1 - \alpha)}{2}$ and
$\sum\limits_{j=1}^n \tilde{p}_{ij} < \frac{c(1 + \alpha)}{2}$. By making use of these two estimates,
we can obtain that
$$
\|\vf'(\vec{u}_0, \vec{v}_0)^{-1}[\vf'(\vec{u},\vec{v}) - \vf'(\vec{u}',\vec{v}')]\|_{\infty}
    < c(1 + \alpha) \left\|\begin{bmatrix} \vec{u} - \vec{u}' \\ \vec{v} - \vec{v}' \end{bmatrix}\right\|_{\infty}.
$$
That is, the Fr\'{e}chet derivative  of $\vf$ satisfies the Lipschitz condition \eqref{cond:LipCond}
with the Lipschitz constant $L = c(1 + \alpha)$. Therefore, Corollary \ref{cor:Conv_LipCond} is applicable to conclude
that the iterative sequence generated by the two step Newton method \eqref{it:TwoStepNM} for nonlinear
operator $\vf$ defined by \eqref{eq:f(u,v)=0} starting from the zero vector $\vec{0}$ converges Q-superquadratically
to the minimal positive solution if
$$
L\beta = c(1 + \alpha) \leq \frac{1}{2}.
$$
Moreover, the order of convergence is cubic at least if $L\beta = c(1 + \alpha) < \frac{4}{9}$.
In particular, we can obtain that the minimal positive $\vec{w}_* := \tran{[\tran{\vec{u}}_*, \tran{\vec{v}}_*]}$ belongs
to the open ball with center $\vec{0}$ and radius $r$, where
$$
\frac{1 - \sqrt{1-2c(1+\alpha)}}{c(1+\alpha)} \leq r < \frac{1 + \sqrt{1-2c(1+\alpha)}}{c(1+\alpha)}.
$$
That is, it satisfies $0 < \|\vec{w}_*\|_{\infty} \leq r$, which coincides with the one given in \cite[Theorem 4.1]{BaiGL2008}.

We end this section with some numerical experiments illustrating the convergence results.
The algorithm for implementing the two step Newton method is summarized by
Algorithm \ref{al:TwoStepNM} as follows.

\begin{algorithm}[h!]
\caption{Two step Newton method for solving
(\ref{eq:f(u,v)=0})} \label{al:TwoStepNM}
Given $c \in (0, 1]$ and $\alpha \in [0,1)$. Choose initial point
$\tran{[\tran{\vec{u}}_0,\tran{\vec{v}}_0]} = \tran{\vec{0}}$.
Form the matrices $P$ and $\tilde{P}$ by \eqref{mat:P_Ptilde}.
For $k = 0,1,2,\ldots$ until convergence, do:
\newcounter{newlist}
\begin{list}{\arabic{newlist}.}{\usecounter{newlist}
\setlength{\rightmargin}{0em}\setlength{\leftmargin}{2.5em}}
\item[Step 1.]
Form the matrix $G(\vec{u}_k, \vec{v}_k)$ by \eqref{mat:G(u,v)}.
\item[Step 2.]
Compute $\tilde{\vec{v}}_k$ from the system of linear equations below:
\begin{align*}
\lefteqn{[I_n - G_2(\vec{u}_k) - H_2(\vec{v}_k)(I_n - G_1(\vec{v}_k))^{-1}
H_1(\vec{u}_k)]\tilde{\vec{v}}_k}\\
& = H_2(\vec{v}_k) (I_n - G_1(\vec{v}_k))^{-1} [\vec{e} - H_1(\vec{u}_k)\vec{v}_k]  + \vec{e} - H_2(\vec{v}_k)\vec{u}_k.
\end{align*}
\item[Step 3.]
Compute $\tilde{\vec{u}}_k$ from the following formula:
$$
\tilde{\vec{u}}_k = (I_n - G_1(\vec{v}_k))^{-1} [H_1(\vec{u}_k) (\tilde{\vec{v}}_k - \vec{v}_k) + \vec{e}].
$$
\item[Step 4.]
Compute $\vec{v}_{k+1}$ from the system of linear equations below:
\begin{align*}
\lefteqn{[I_n - G_2(\vec{u}_k) - H_2(\vec{v}_k)(I_n - G_1(\vec{v}_k))^{-1} H_1(\vec{u}_k)]\vec{v}_{k+1}}\\
& = H_2(\vec{v}_k) (I_n - G_1(\vec{v}_k))^{-1} [\tilde{\vec{u}}_k \circ P(\tilde{\vec{v}}_k
    - \vec{v}_k) - H_1(\vec{u}_k)\tilde{\vec{v}}_k + \vec{e}] \\
& \quad  + \tilde{\vec{v}}_k \circ \tilde{P}(\tilde{\vec{u}}_k - \vec{u}_k) - H_2(\vec{v}_k)\tilde{\vec{u}}_k + \vec{e}.
\end{align*}
\item[Step 5.]
Compute $\vec{u}_{k+1}$ from the following formula:
\begin{align*}
\vec{u}_{k+1} & = (I_n - G_1(\vec{v}_k))^{-1}
[\tilde{\vec{u}}_k \circ P(\tilde{\vec{v}}_k - \vec{v}_k) + \vec{e} + H_1(\vec{u}_k)(\vec{v}_{k+1} -
\tilde{\vec{v}}_k)].
\end{align*}
\end{list}
\end{algorithm}

As Example 5.2 of \cite{GuoL2000b}, the
constants $c_i$ and $\omega_i$ are given by a numerical quadrature
formula on the interval $[0,1]$ which is obtained by dividing
$[0,1]$ into $n/4$ subintervals of equal length and applying
Gauss-Legendre quadrature with four nodes to each subinterval.
Our numerical experiments were carried out in MATLAB version R2014a running on
a PC with Intel(R) Core(TM) i3-3110M of 2.40 GHz CPU and 12GB memory.
In our implementations, the iterations in the algorithm are stopped when
the following condition is satisfied:
\begin{equation*}
\label{cond:StopCond}
\text{Res} := \max\left\{\frac{\|\vec{u}_{k+1} - \vec{u}_k\|_{\infty}}{\|\vec{u}_{k+1}\|_{\infty}},
    \frac{\|\vec{v}_{k+1} - \vec{v}_k\|_{\infty}}{\|\vec{v}_{k+1}\|_{\infty}}\right\}
\leq \frac{\sqrt{n}}{2}\cdot \mbox{eps},
\end{equation*}
where $n$ the size of matrix $D$ given
by (\ref{mat:ABCD}) and $\mbox{eps} = 2^{-52} \approx 2.2204 \times 10^{-16}$.
The CPU time (in seconds) is computed by using the MATLAB function \textsf{cputime}.
In each test, we run the same program 10 times and choose the average time as the time
spent by the algorithm. Moreover, we use ``iter'' to stand for the number of the iterations needed.

Figure \ref{fig:ConvHistory} depicts the convergence histories
for the problem size $n = 1024, 2048$ and $4096$ with six different pairs of $(\alpha, c)$, respectively,
namely, $(0.5,1/3)$, $(0.5,2/9)$, $(0.5,1/9)$, $(0.25,2/5)$, $(0.25,1/3)$ and $(0.25,1/10)$.
As one can see in the figure, for each case, the speedup is obtained as the value $L\beta = c(1+\alpha)$ decreases.
More convergence results including the number of iterations, the relative residual and the CPU time
for various problem size $n$ are listed
Tables \ref{tab:n1024}, \ref{tab:n2048} and \ref{tab:n4096}, respectively.
Obviously, we see from these tables that it requires less time when the  value $L\beta = c(1+\alpha)$ is taken smaller.

\begin{figure}[h!]
\centering
\subfigure[]{\includegraphics[width=0.46\textwidth]{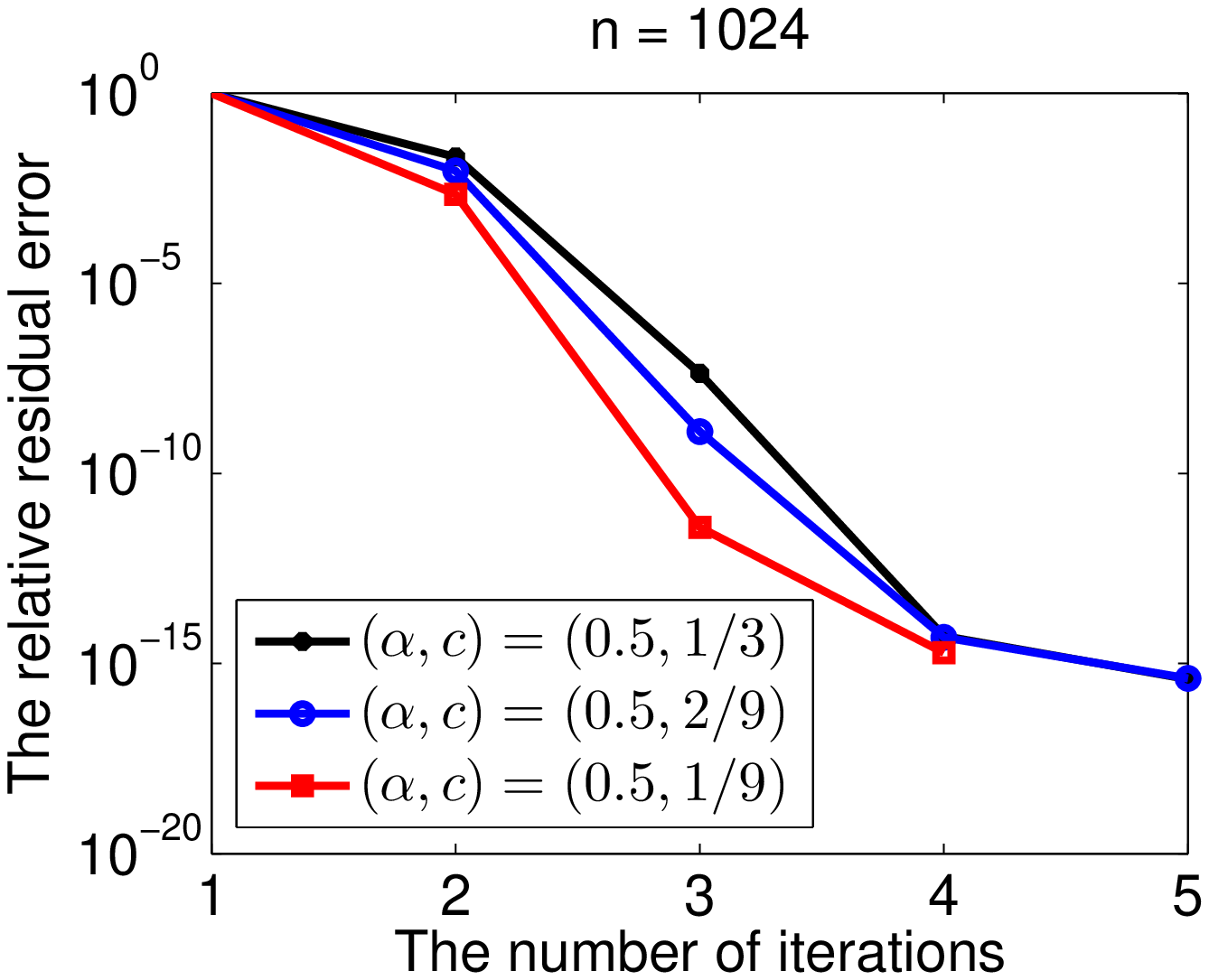}}\qquad
\subfigure[]{\includegraphics[width=0.46\textwidth]{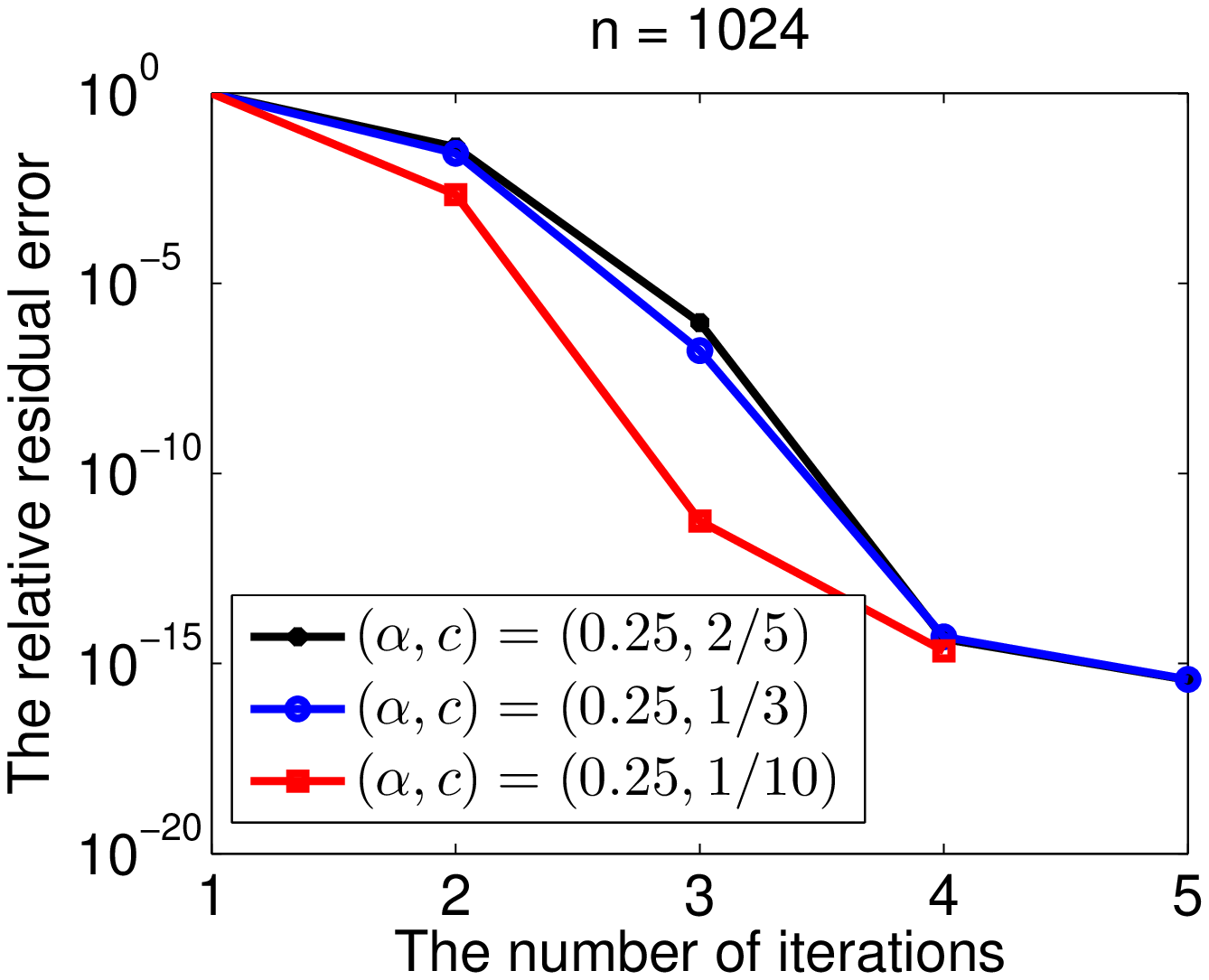}} \\
\subfigure[]{\includegraphics[width=0.46\textwidth]{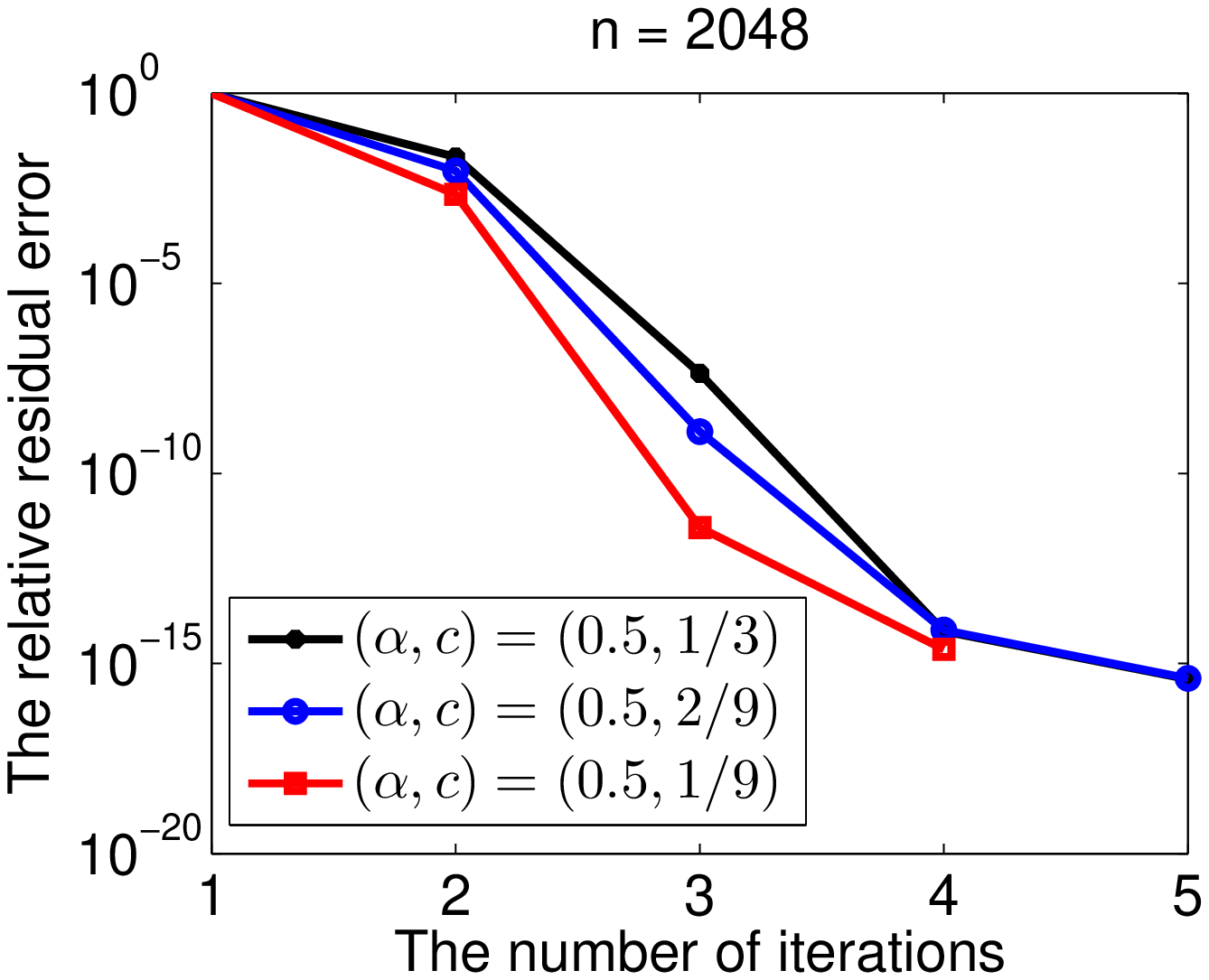}}\qquad
\subfigure[]{\includegraphics[width=0.46\textwidth]{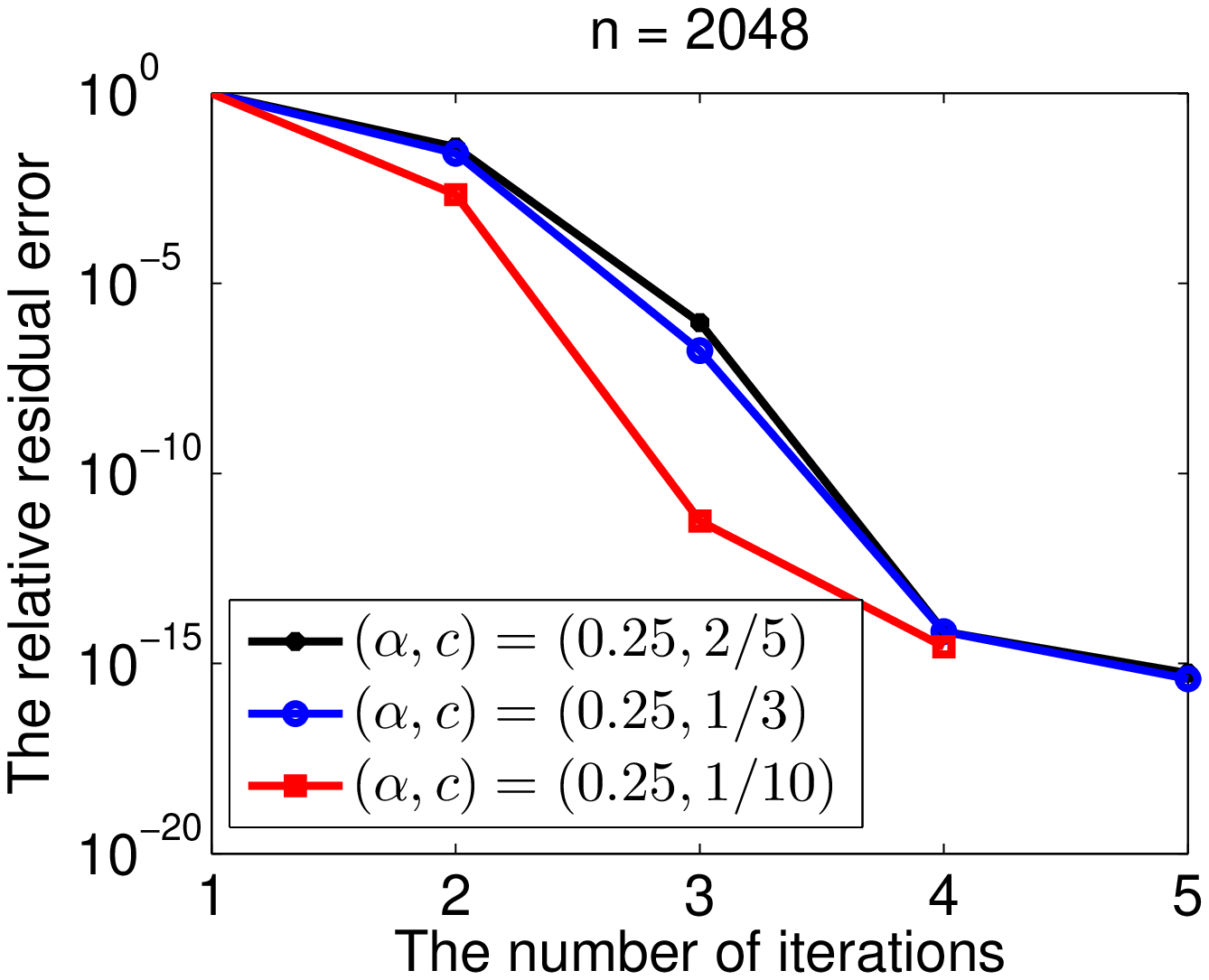}} \\
\subfigure[]{\includegraphics[width=0.46\textwidth]{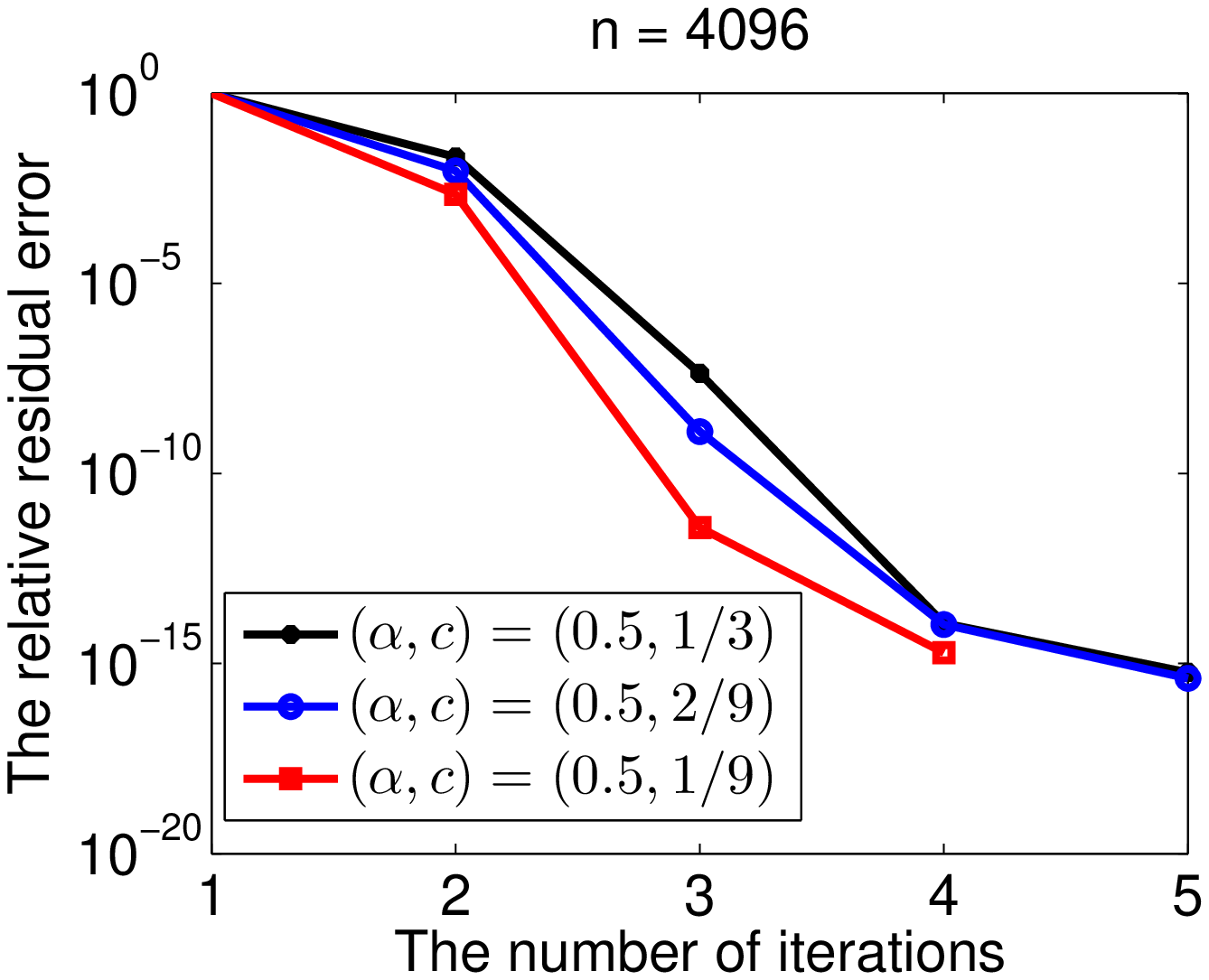}}\qquad
\subfigure[]{\includegraphics[width=0.46\textwidth]{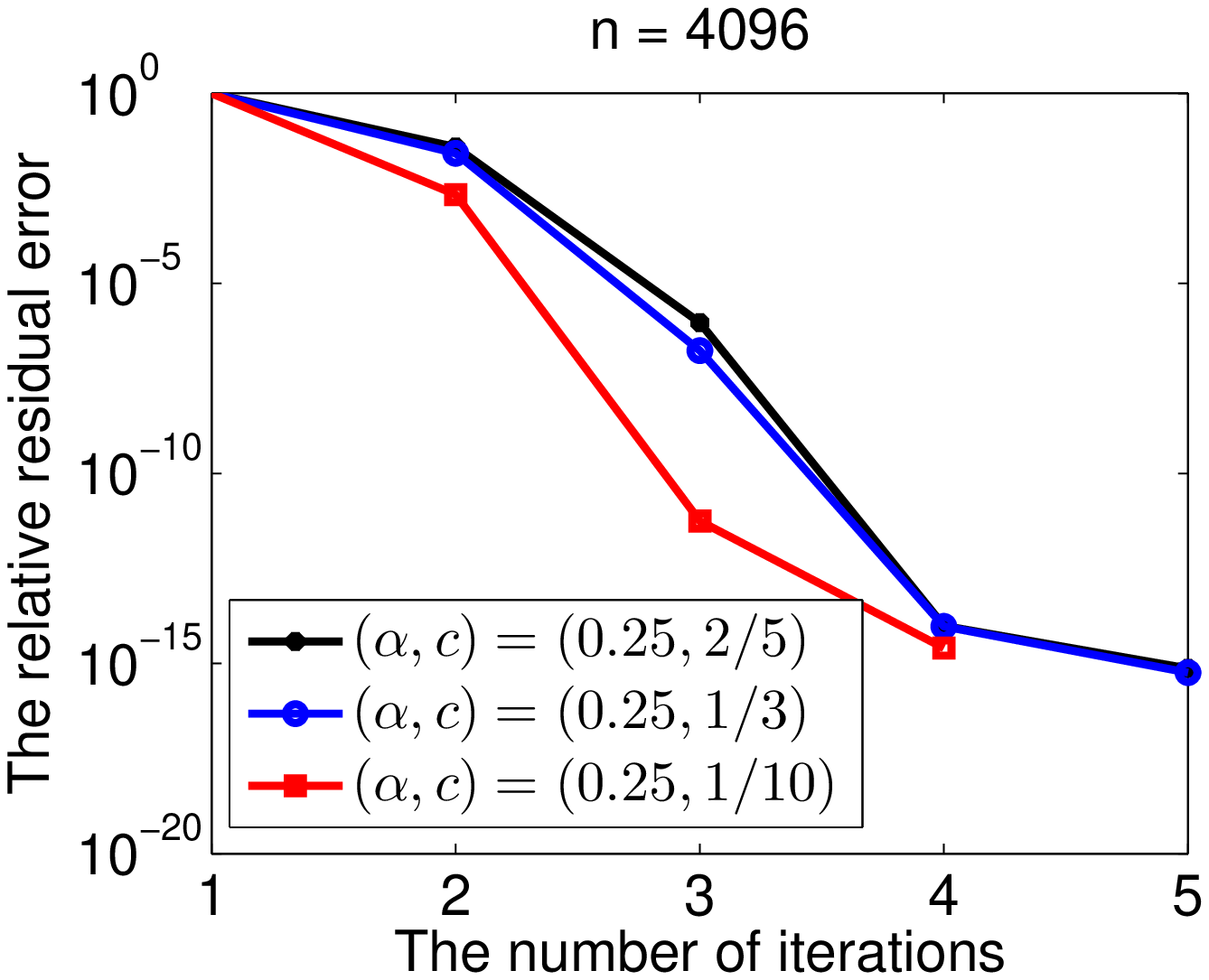}} \\
\caption{The convergence histories for various $(\alpha, c)$ when the problem size $n = 1024, 2048, 4096$, respectively.}\label{fig:ConvHistory}
\end{figure}

\begin{table}[h!]
\caption{The results for the problem size $n = 1024$} \label{tab:n1024}
\begin{center}
\begin{tabular}{ll@{\hspace{0.7cm}}l@{\hspace{0.7cm}}l@{\hspace{0.7cm}}l@{\hspace{0.7cm}}l}
\hline
& $(\alpha, c)$ & $L\beta$ & iter & Res & CPU time (s) \\
\hline
& $(0.5, 1/3)$ & 1/2 & 5 & 3.8969e-16 & 7.6050  \\
& $(0.5, 2/9)$ & 1/3 & 5 & 4.0858e-16 & 7.4600  \\
& $(0.5, 1/9)$ & 1/6 & 4 & 1.9201e-15 & 5.9390  \\
& $(0.25, 2/5)$ & 1/2 & 5 & 3.7210e-16 & 7.4178  \\
& $(0.25, 1/3)$ & 5/12 & 5 & 3.8521e-16 & 7.3523  \\
& $(0.25, 1/10)$ & 1/8 & 4 & 2.1370e-15 & 5.9062  \\
\hline
\end{tabular}
\end{center}
\end{table}

\begin{table}[h!]
\caption{The results for the problem size $n = 2048$} \label{tab:n2048}
\begin{center}
\begin{tabular}{ll@{\hspace{0.7cm}}l@{\hspace{0.7cm}}l@{\hspace{0.7cm}}l@{\hspace{0.7cm}}l}
\hline
& $(\alpha, c)$ & $L\beta$ & iter & Res & CPU time (s) \\
\hline
& $(0.5, 1/3)$ & 1/2 & 5 & 3.8969e-16 & 50.6707  \\
& $(0.5, 2/9)$ & 1/3 & 5 & 4.0858e-16 & 49.3416  \\
& $(0.5, 1/9)$ & 1/6 & 4 & 2.3467e-15 & 40.7350  \\
& $(0.25, 2/5)$ & 1/2 & 5 & 5.5815e-16 & 52.0825  \\
& $(0.25, 1/3)$ & 5/12 & 5 & 3.9646e-16 & 51.6239  \\
& $(0.25, 1/10)$ & 1/8 & 4 & 2.7781e-15 & 40.7989  \\
\hline
\end{tabular}
\end{center}
\end{table}

\begin{table}[h!]
\caption{The results for the problem size $n = 4096$} \label{tab:n4096}
\begin{center}
\begin{tabular}{ll@{\hspace{0.7cm}}l@{\hspace{0.7cm}}l@{\hspace{0.7cm}}l@{\hspace{0.7cm}}l}
\hline
& $(\alpha, c)$ & $L\beta$ & iter & Res & CPU time (s) \\
\hline
& $(0.5, 1/3)$ & 1/2 & 5 & 5.8453e-16 & 351.6434  \\
& $(0.5, 2/9)$ & 1/3 & 5 & 4.0858e-16 & 350.8556  \\
& $(0.5, 1/9)$ & 1/6 & 4 & 1.9200e-15 & 290.0932  \\
& $(0.25, 2/5)$ & 1/2 & 5 & 7.4419e-16 & 363.0767  \\
& $(0.25, 1/3)$ & 5/12 & 5 & 5.7781e-16 & 353.1909  \\
& $(0.25, 1/10)$ & 1/8 & 4 & 2.5644e-15 & 275.7286  \\
\hline
\end{tabular}
\end{center}
\end{table}

In conclusion, the above numerical experiments confirm our convergence results stated in Corollary \ref{cor:Conv_LipCond}
for the two step Newton method \eqref{it:TwoStepNM}.

\section{Conclusions}

In this paper, we have presented the semilocal convergence analysis for two-step Newton method
\eqref{it:TwoStepNM} under the assumption that the first derivative $F'$ satisfies the
$L$-average Lipschitz conditions \eqref{cond:LAverageLipCond} in Banach spaces.
The main results are contained in Theorem \ref{th:SemilocalConv}.
When the unified convergence criteria $0< \beta \leq b$ given by Wang in \cite{Wang1999} is satisfied,
the existence and uniqueness of a solution $x^* \in \ball(x_0, r)$ of \eqref{eq:F(x)=0} are shown,
and the superquadratic convergence of the sequence $\{x_k\}$ is also proved.
Moreover, we proved that the sequence $\{x_k\}$ is Q-cubically convergent if
the condition \eqref{cond:CubicConv_Cond} is satisfied additionally.
Three special cases which include the Kantorovich type conditions,
$\gamma$-conditions and Nesterov-Nemirovskii conditions have been provided.
We also have applied our convergence result to solve the approximation of minimal
positive solution for a nonsymmetric algebraic Riccati equation arising from transport theory.
One goal of our future research is to exploit this general theory to develop more practical and efficient
two-step inexact Newton method for solving IEP that can easily capture practical applications in
large-scale settings.


\begin{thebibliography}{10}

\bibitem{AlvarezBM2008}
{\sc F.~Alvarez, J.~Bolte, and J.~Munier}, {\em A unifying local convergence
  result for {Newton}'s method in {Riemannian} manifolds}, Found. Comput.
  Math., 8 (2008), pp.~197--226.

\bibitem{AmatBG2011}
{\sc S.~Amat, S.~Busquier, and J.~M. Guti{\'e}rrez}, {\em {Third-order
  iterative methods with applications to Hammerstein equations: a unified
  approach}}, J. Comput. Appl. Math., 235 (2011), pp.~2936--2943.

\bibitem{AppellDEZ1995}
{\sc J.~Appell, E.~De~Pascale, N.~A. Evkhuta, and P.~P. Zabrejko}, {\em {On the
  two-step Newton method for the solution of nonlinear operator equations}},
  Math. Nachr., 172 (1995), pp.~5--14.

\bibitem{Argyros2001}
{\sc I.~K. Argyros}, {\em {A new semilocal convergence theorem for Newton's
  method in Banach space using hypotheses on the second
  Fr{\'e}chet-derivative}}, J. Comput. Appl. Math., 130 (2001), pp.~369--373.

\bibitem{ArgyrosH2011b}
{\sc I.~K. Argyros and S.~Hilout}, {\em {Extending the applicability of the
  Gauss--Newton method under average Lipschitz--type conditions}}, Numer.
  Algorithms, 58 (2011), pp.~23--52.

\bibitem{ArgyrosK2015}
{\sc I.~K. Argyros and S.~K. Khattri}, {\em {Weak convergence conditions for
  the Newton’s method in Banach space using general majorizing sequences}},
  Appl. Math. Comput., 263 (2015), pp.~59--72.

\bibitem{BaiGL2008}
{\sc Z.~Bai, Y.~Gao, and L.~Lu}, {\em {Fast iterative schemes for nonsymmetric
  algebraic Riccati equations arising from transport theory}}, SIAM J. Sci.
  Comput., 30 (2008), pp.~804--818.

\bibitem{BittencourtF2017}
{\sc T.~Bittencourt and O.~P. Ferreira}, {\em {Kantorovich’s theorem on
  Newton’s method under majorant condition in Riemannian manifolds}}, J.
  Global Optim., 68 (2017), pp.~387--411.

\bibitem{CandelaM1990a}
{\sc V.~Candela and A.~Marquina}, {\em {Recurrence relations for rational cubic
  methods I: The Halley method}}, Computing, 44 (1990), pp.~169--184.

\bibitem{CandelaM1990b}
{\sc V.~Candela and A.~Marquina}, {\em {Recurrence relations for rational cubic
  methods II: The Chebyshev method}}, Computing, 45 (1990), pp.~355--367.

\bibitem{Chen1994}
{\sc P.~Chen}, {\em Approximate zeros of quadratically convergent algorithms},
  Math. Comput., 63 (1994), pp.~247--270.

\bibitem{ChenWS2018}
{\sc X.~Chen, C.~Wen, and H.-W. Sun}, {\em {Two-step Newton-type methods for
  solving inverse eigenvalue problems}}, Numer. Linear Algebra Appl.,  (2018),
  p.~e2185.
\newblock to appear.

\bibitem{Dedieu2000}
{\sc J.~P. Dedieu}, {\em Approximate solutions of analytic inequality systems},
  SIAM J. Optim., 11 (2000), pp.~411--425.

\bibitem{DedieuK2002}
{\sc J.~P. Dedieu and M.~H. Kim}, {\em {Newton's} method for analytic systems
  of equations with constant rank derivatives}, J. Complexity, 18 (2002),
  pp.~187--209.

\bibitem{DedieuPM2003}
{\sc J.-P. Dedieu, P.~Priouret, and G.~Malajovich}, {\em {Newton}'s method on
  {Riemannian} manifolds: Covariant alpha theory}, IMA J. Numer. Anal., 23
  (2003), pp.~395--419.

\bibitem{DedieuS2000a}
{\sc J.~P. Dedieu and M.~Shub}, {\em Multihomogeneous {Newton} methods}, Math.
  Comput., 69 (2000), pp.~1071--1098.

\bibitem{DedieuS2000b}
{\sc J.~P. Dedieu and M.~Shub}, {\em {Newton}'s method for overdetermined
  systems of equations}, Math. Comput., 69 (2000), pp.~1099--1115.

\bibitem{DeuflhardH1979}
{\sc P.~Deuflhard and G.~Heindl}, {\em Affine invariant convergence theorems
  for {Newton's} method and extensions to related methods}, SIAM J. Numer.
  Anal., 16 (1979), pp.~1--10.

\bibitem{EzquerroH2002}
{\sc J.~A. Ezquerro and M.~A. Hern\'{a}ndez}, {\em Generalized
  differentiability conditions for {Newton's} method}, IMA J. Numer. Anal., 22
  (2002), pp.~519--530.

\bibitem{EzquerroH2005}
{\sc J.~A. Ezquerro and M.~A. Hern\'{a}ndez}, {\em On the {R}-order of the
  {Halley} method}, J. Math. Anal. Appl., 303 (2005), pp.~591--601.

\bibitem{EzquerroH2009a}
{\sc J.~A. Ezquerro and M.~A. Hern\'{a}ndez}, {\em {An improvement of the
  region of accessibility of Chebyshev's method from Newton's method}}, Math.
  Comput., 78 (2009), pp.~1613--1627.

\bibitem{EzquerroH2009b}
{\sc J.~A. Ezquerro and M.~A. Hern\'{a}ndez}, {\em {An optimization of
  Chebyshev’s method}}, J. Complexity, 25 (2009), pp.~343--361.

\bibitem{EzquerroHM2018}
{\sc J.~A. Ezquerro, M.~A. Hern\'{a}ndez, and A.~A. Magre\~{n}\'{a}n}, {\em
  {Starting points for Newton’s method under a center Lipschitz condition for
  the second derivative}}, J. Comput. Appl. Math., 330 (2018), pp.~721--731.

\bibitem{ferreira2015}
{\sc O.~P. Ferreira}, {\em {A robust semi-local convergence analysis of
  Newton’s method for cone inclusion problems in Banach spaces under affine
  invariant majorant condition}}, J. Comput. Appl. Math., 279 (2015),
  pp.~318--335.

\bibitem{FerreiraS2009}
{\sc O.~P. Ferreira and B.~F. Svaiter}, {\em {Kantorovich}'s majorants
  principle for {Newton's} method}, Comput. Optim. Appl., 42 (2009),
  pp.~213--229.

\bibitem{FerreiraS2012}
{\sc O.~P. Ferreira and B.~F. Svaiter}, {\em {A robust Kantorovich’s theorem
  on the inexact Newton method with relative residual error tolerance}},
  Journal of Complexity, 28 (2012), pp.~346--363.

\bibitem{GraggT1974}
{\sc W.~B. Gragg and R.~A. Tapia}, {\em Optimal error bounds for the
  {Newton}-{Kantorovich} theorem}, SIAM J. Numer. Anal., 11 (1974), pp.~10--13.

\bibitem{GuoL2000b}
{\sc C.-H. Guo and A.~J. Laub}, {\em {On the iterative solution of a class of
  nonsymmetric algebraic Riccati equations}}, SIAM J. Matrix Anal. Appl., 22
  (2000), pp.~376--391.

\bibitem{GutierrezH1997}
{\sc J.~M. Guti\'{e}rrez and M.~A. Hern\'{a}ndez}, {\em A family of
  {Chebyshev}-{Halley} type methods in {Banach} spaces}, Bull. Aust. Math.
  Soc., 55 (1997), pp.~113--130.

\bibitem{GutierrezH2000}
{\sc J.~M. Guti\'{e}rrez and M.~A. Hern\'{a}ndez}, {\em {Newton's} method under
  weak {Kantorovich} conditions}, IMA J. Numer. Anal., 20 (2000), pp.~521--532.

\bibitem{Han2001}
{\sc D.~Han}, {\em The convergence on a family of iterations with cubic order},
  J. Comput. Math., 19 (2001), pp.~467--474.

\bibitem{HanW1997}
{\sc D.~Han and X.~Wang}, {\em The error estimates of {Halley's} method},
  Numer. Math. JCU (Engl. Ser.), 6 (1997), pp.~231--240.

\bibitem{HiriartL1993}
{\sc J.-B. Hiriart-Urruty and C.~Lemar{\'e}chal}, {\em {Convex Analysis and
  Minimization Algorithms I: Fundamentals}}, vol.~305, Springer science \&
  business media, 1993.

\bibitem{Huang1993}
{\sc Z.~Huang}, {\em A note on the {Kantorovich} theorem for {Newton}
  iteration}, J. Comput. Appl. Math., 47 (1993), pp.~211--217.

\bibitem{Iannazzo2008}
{\sc B.~Iannazzo}, {\em A family of rational iterations and its application to
  the computation of the matrix $p$th root}, SIAM J. Matrix Anal. Appl., 30
  (2008), pp.~1445--1462.

\bibitem{Jay2001}
{\sc L.~O. Jay}, {\em {A note on Q-order of convergence}}, BIT Numer. Math., 41
  (2001), pp.~422--429.

\bibitem{Juang1995}
{\sc J.~Juang}, {\em {Existence of algebraic matrix Riccati equations arising
  in transport theory}}, Linear Algebra Appl., 230 (1995), pp.~89--100.

\bibitem{KantorvichA1982}
{\sc L.~V. Kantorvich and G.~P. Akilov}, {\em {Functional Analysis, second
  edition}}, Pergamon Press, Oxford, 1982.

\bibitem{Kelley2018}
{\sc C.~T. Kelley}, {\em Numerical methods for nonlinear equations}, Acta
  Numer., 27 (2018), pp.~207--287.

\bibitem{LiHW2010}
{\sc C.~Li, N.~Hu, and J.~Wang}, {\em Convergence behavior of {Gauss-Newton's}
  method and extensions of the {Smale} point estimate theory}, J. Complexity,
  26 (2010), pp.~268--295.

\bibitem{LiN2007}
{\sc C.~Li and K.~Ng}, {\em {Majorizing functions and convergence of the
  Gauss-Newton for convex composite optimization}}, SIAM J. Optim., 18 (2007),
  pp.~613--642.

\bibitem{LiN2013}
{\sc C.~Li and K.~Ng}, {\em {Approximate solutions for abstract inequality
  systems}}, SIAM J. Optim., 23 (2013), pp.~1237--1256.

\bibitem{LiN2016}
{\sc C.~Li and K.~Ng}, {\em {Extended Newton methods for conic inequalities:
  Approximate solutions and the extended Smale $\alpha$-theory}}, J. Math.
  Anal. Appl., 440 (2016), pp.~636--660.

\bibitem{LiN2018}
{\sc C.~Li and K.~Ng}, {\em {Quantitative analysis for perturbed abstract
  inequality systems in Banach spaces}}, SIAM J. Optim., 28 (2018),
  pp.~2872--2901.

\bibitem{LiW2006}
{\sc C.~Li and J.~Wang}, {\em {Newton's} method on {Riemannian} manifolds:
  {Smale}'s point estimate theory under the $\gamma-$condition}, IMA J. Numer.
  Anal., 26 (2006), pp.~228--251.

\bibitem{LinB2008}
{\sc Y.~Lin and L.~Bao}, {\em {Convergence analysis of the Newton--Shamanskii
  method for a nonsymmetric algebraic Riccati equation}}, Numer. Linear Algebra
  Appl., 15 (2008), pp.~535--546.

\bibitem{LingH2017}
{\sc Y.~Ling and Z.~Huang}, {\em {An analysis on the efficiency of Euler's
  method for computing the matrix pth root}}, Numer. Linear Algebra Appl., 24
  (2017), p.~e2104.

\bibitem{LingX2014}
{\sc Y.~Ling and X.~Xu}, {\em {On the semilocal convergence behavior for
  Halley’s method}}, Comput. Optim. Appl., 58 (2014), pp.~597--618.

\bibitem{LingX2017}
{\sc Y.~Ling and X.~Xu}, {\em {On one-parameter family of Newton-like
  iterations for solving nonsymmetric algebraic Riccati equation from transport
  theory}}, J. Nonlinear Convex Anal., 18 (2017), pp.~1833--1848.

\bibitem{Lu2005a}
{\sc L.-Z. Lu}, {\em {Solution form and simple iteration of a nonsymmetric
  algebraic Riccati equation arising in transport theory}}, SIAM J. Matrix
  Anal. Appl., 26 (2005), pp.~679--685.

\bibitem{ArgyrosMR2014}
{\sc A.~A. {Magre\~{n}\'{a}n Ruiz} and I.~K. Argyros}, {\em {Two-step Newton
  methods}}, J. Complexity, 30 (2014), pp.~533--553.

\bibitem{NakatsukasaBG2010}
{\sc Y.~Nakatsukasa, Z.~Bai, and F.~Gygi}, {\em {Optimizing Halley's iteration
  for computing the matrix polar decomposition}}, SIAM J. Matrix Anal. Appl.,
  31 (2010), pp.~2700--2720.

\bibitem{NakatsukasaF2016}
{\sc Y.~Nakatsukasa and R.~W. Freund}, {\em {Computing fundamental matrix
  decompositions accurately via the matrix sign function in two iterations: The
  power of Zolotarev's functions}}, SIAM Rev., 58 (2016), pp.~461--493.

\bibitem{NesterovN1994}
{\sc Y.~Nesterov and A.~Nemirovskii}, {\em {Interior-point Polynomial
  Algorithms in Convex Programming}}, SIAM, Philadelphia, 1994.

\bibitem{Potra1989}
{\sc F.~A. Potra}, {\em {On Q-order and R-order of convergence}}, J. Optim.
  Theory Appl., 63 (1989), pp.~415--431.

\bibitem{ShubS1996}
{\sc M.~Shub and S.~Smale}, {\em {Complexity of B\'{e}zout's theorem IV:
  Probability of success, extensions}}, SIAM J. Numer. Anal., 33 (1996),
  pp.~128--148.

\bibitem{Smale1986}
{\sc S.~Smale}, {\em {Newton's} method estimates from data at one point}, in
  The Merging of Disciplines: New Directions in Pure, Applied and Computational
  Mathematics, R.~Ewing, K.~Gross, and C.~Martin, eds., Springer-Verlag, New
  York, 1986, pp.~185--196.

\bibitem{Smale1987}
{\sc S.~Smale}, {\em The fundamental theory for solving equations}, in
  Proceeding of the International Congress of Mathematicians, AMS, Providence,
  RI, 1987, p.~185.

\bibitem{Smale1997}
{\sc S.~Smale}, {\em Complexity theory and numerical analysis}, Acta Numer., 6
  (1997), pp.~523--552.

\bibitem{Wang1999}
{\sc X.~Wang}, {\em Convergence of {Newton's} method and inverse functions
  theorem in {Banach} space}, Math. Comput., 68 (1999), pp.~169--186.

\bibitem{WangH1990}
{\sc X.~Wang and D.~Han}, {\em On dominating sequence method in the point
  estimate and {Smale} theorem}, Sci. China Ser. A, 33 (1990), pp.~135--144.

\bibitem{WangH1997a}
{\sc X.~Wang and D.~Han}, {\em Criterion $\alpha$ and {Newton's} method under
  weak conditions}, Chinese J. Numer. Appl. Math., 19 (1997), pp.~96--105.

\bibitem{WangL2001}
{\sc X.~Wang and C.~Li}, {\em Local and global behavior for algorithms of
  solving equations}, Chinese Sci. Bull., 46 (2001), pp.~444--451.

\bibitem{XuL2008}
{\sc X.~Xu and C.~Li}, {\em Convergence criterion of {Newton's} method for
  singular systems with constant rank derivatives}, J. Math. Anal. Appl., 345
  (2008), pp.~689--701.

\bibitem{ZabrejkoN1987}
{\sc P.~P. Zabrejko and D.~F. Nguen}, {\em {The majorant method in the theory
  of Newton-Kantorovich approximations and the Ptak error estimates}}, Numer.
  Funct. Anal. Optim., 9 (1987), pp.~671--684.

\end{thebibliography}


\end{document}